		\def\version{16 December 2022}				       % 
\documentclass[reqno,11pt]{amsart} 
\usepackage{amsmath,amsthm,amssymb,a4,graphics,color}
\usepackage[active]{srcltx}
\usepackage[usenames,dvipsnames]{pstricks}
\usepackage{epsfig}
\usepackage{pst-grad}
\usepackage{pst-plot} 
\usepackage{tikz}
\usepackage{hyperref}

\newfam\Bbbfam 
\font\tenBbb=msbm10 
\font\sevenBbb=msbm7 
\font\fiveBbb=msbm5 
\textfont\Bbbfam=\tenBbb 
\scriptfont\Bbbfam=\sevenBbb 
\scriptscriptfont\Bbbfam=\fiveBbb

\newtheorem{theorem}{Theorem}[section] 
\newtheorem{lemma}[theorem]{Lemma} 
\newtheorem{prop}[theorem] {Proposition} 
\newtheorem{cor}[theorem]  {Corollary} 
\newtheorem{remark}[theorem]  {Remark}

\theoremstyle{definition}

\newcommand{\E}{\mathbb E}

\newcommand{\R}{\mathbb R}
\newcommand{\N}{\mathbb N}

\renewcommand{\P}{\mathbb P}

\newcommand{\smfrac}[2]{\textstyle{\frac {#1}{#2}}}
\def\1{{\mathchoice {1\mskip-4mu\mathrm l}      % Blackboard bold 1 
{1\mskip-4mu\mathrm l} 
{1\mskip-4.5mu\mathrm l} {1\mskip-5mu\mathrm l}}} 
\newcommand{\ssup}[1] {{\scriptscriptstyle{({#1})}}}

\renewcommand{\qed}{\hfill\ensuremath{\square}}

\renewcommand{\d}{{\rm d}} 
 
\newcommand{\eps}{\varepsilon}

\newcommand{\Poi}{{\operatorname {Poi}}}

\newcommand{\Mcal}   {{\mathcal M }} 
\newcommand{\Ncal}   {{\mathcal N }}

\newcommand{\Xcal}   {{\mathcal X }}

\newcommand{\e}   {{\operatorname e }}

\numberwithin{equation}{section}

\begin{document}
\title[Markovian description of multi-channel ALOHA and CSMA]{Multi-channel ALOHA and CSMA medium-access protocols:\\\medskip
Markovian description and large deviations}
\author[Wolfgang K\"onig and Helia Shafigh]{}
\maketitle
\thispagestyle{empty}
\vspace{-0.5cm}

\centerline{\sc 
Wolfgang K\"onig\footnote{TU Berlin and WIAS Berlin, Mohrenstra{\ss}e 39, 10117 Berlin, Germany, {\tt koenig@wias-berlin.de}} and Helia Shafigh\footnote{WIAS Berlin, Mohrenstra{\ss}e 39, 10117 Berlin, Germany, {\tt shafigh@wias-berlin.de}}}
\renewcommand{\thefootnote}{}
\vspace{0.5cm}
\centerline{\textit{WIAS Berlin and TU Berlin, and WIAS Berlin}}

\bigskip

\centerline{\small(\version)} 
\vspace{.5cm}

\begin{abstract}
We consider a multi-channel communication system under ALOHA and CSMA protocols, resepctively, in continuous time. We derive probabilistic formulas for the most important quantities: the numbers of sending attempts and the number of successfully delivered messages in a given time interval. We derive (1) explicit formulas for the large-time limiting throughput, (2) introduce an explicit and ergodic Markov chain for a deeper probabilistic analysis, and use this to (3) derive exponential asymptotics for rare events for these quantities in the limit of large time, via large-deviation principles. 
 \end{abstract}

\vspace{.2cm}


\section{Introduction and main results}\label{sec-Intro}

Protocols for medium access control (MAC) are fundamental and ubiquitous in any telecommunication system. Here we are particularly interested in {\em multi-channel systems}, where a fixed number of channels is available. Our goal is to develop a probabilistic model for the ALOHA and the CSMA protocol, which can be easily realised on a computer and be mathematically analysed in explicit terms. We will describe the stochastic process of arrival times of incoming and successfully delivered messages and the times that elapse in between in terms of a Markov renewal process that is explicit and has very good ergodic properties.

\subsection{Medium access protocols; our goals}\label{sec-background}

We consider the {\em ALOHA protocol}, where no infra\-structure is given and collisions are possible, and the  {\it Carrier Sense Multiple Access (CSMA) protocol,} where a message is delivered only if there is an idle channel at the time  when the message arrives. We work in continuous time and assume that the messages arrive at random times that are given by a Poisson point process (PPP). To keep things simple, we assume that the service times (delivery times) are all equal to one. 

We are interested in a {\em multi--channel system}, i.e., we assume that at most $\kappa$ messages can be processed at any given time, where $\kappa\in\N$ is a parameter. We think of two interpretations of this restriction: Either there are $\kappa$ channels available in our system that can be used independently all the time, or there are interference constraints that make it impossible that more than $\kappa$ messages can be transmitted at the same time, and any additional message is refused from the system. 

Our interest lies on important quantities like the total number of messages in the system, the number of incoming messages and the number of successfully delivered messages in a given fixed time interval. We strive to calculate expected values (in the limit of large time intervals), which follows elementary ideas, but also to analyse more detailed questions, like probabilities of certain events, which needs a deeper understanding of the communication system.

To this sake, as one of our main novelties, we develop a description in terms of an {\em explicit Markov chain} (in discrete time) that admits a description of the mentioned quantities as functionals of this chain in terms of a kind of a {\em Markov renewal process}. To the best of our knowledge, such a Markov chain was not yet known in comparable situations, even though there are a number of ansatzes with queues and the related theory; however these stochastic processes are not able to give information about the real time, but only about certain quantities (i.e., number of messages in the system) at the (random) arrival times or the (random) delivery times of the messages. Our Markov chain  makes the application of a number of well-known {\em probabilistic tools} available, like invariant initial distributions, ergodic theory and large deviations theory. We give explicit formulas for the transition probabilities and prove that this chain is uniformly ergodic, hence this Markov chain is also useful for making computer simulations.

Our Markov chain in particular opens the possibility to describe (the probabilities of) rare events, which is for example helpful if one wants to understand ubiquitous situations in which the system underachieves by producing a smaller throughput than is expected over a long time stretch. The probabilistic  {\em theory of large deviations} provides mathematical tools for deriving formulas for the exponential decay of the probability, and it provides also tools for characterising the most likely behaviour of the system in this unlikely situation. For the application of this theory, one needs a powerful description and a high degree or ergodicity, and this is provided by our Markov chain. Unfortunately, our description does not allow for the determination of sharp exponential lower bounds for the probabilities of large deviations, but only exponential upper bounds. But we believe it is the exponential upper bounds that shows the value of the large-deviation theory for understanding such communication system.

The most important parameters in the system will be the number of channels, $\kappa\in\N$, and the density parameter of the incoming messages, $\lambda\in(0,\infty)$. We assume that the arrival times follow a standard Poisson point process with parameter $\lambda$. We will conceive the situation from the user's perspective and will discuss the optimal value of $\lambda$ for achieving a {\em maximal throughput}. The idea behind this is that each user has the knowledge about the number of users in a vicinity of the $\kappa$ channels and assumes that each of them makes message sending attempts at a certain rate that amounts to the total rate $\lambda$ of all the message attempts. Under these assumptions, the optimal value of $ \lambda$, divided by the number of users,  should then be the probability parameter for making sending attempts.

In the CSMA setting, it will be clear that the throughput is an increasing function of $\lambda$,  and the optimisation is trivial (ignoring a potential trade-off coming from a huge number of unsuccessful messages). However, for the ALOHA setting, it will be interesting to identify the optimal density $\lambda$ (depending on $\kappa$) for having a maximal throughput; this will be one of our results. 

Summarizing, the main contributions of this paper are the following.

\begin{itemize}
\item An explicit probabilistic description (in terms of a Markov renewal process) of the number of messages and the number of successfully delivered messages and more quantities at a deterministic time,

\item explicit formulas for limiting expected values of these quantities and optimal values of parameters,

\item a large-deviations analysis of rare events involving these quantities.

\end{itemize}

\subsection{Description of the models}\label{sec-slottedALOHA}

We consider a system with a steady flow of incoming messages that require access to the system at random times. The time lags between any two subsequent arrivals of two messages are independent exponentially distributed times with density parameter $\lambda\in(0,\infty)$. That is, the sequence of arrival times forms a standard Poisson point process (PPP) with parameter $\lambda$. Any successful message transmission has a duration of precisely one time unit; i.e., each service time is equal to one.

We assume that $\kappa$ channels are available. On arrival, each message asks for access to some of them. Now we consider two different algorithms (medium access protocols) according to which this request is handled:

\begin{itemize}
\item[$\bullet$] {\it ALOHA:}  For the message, one of the $\kappa$ channels is picked uniformly at random. All these channel choices are independent over all messages.
\begin{itemize}
\item If the channel is already busy i.e., if the transmission of another message in this channel is still running, then the new incoming message collides with the old one, causing the cancellation of both messages. 

\item If the channel is idle, the new message is admitted immediately. If not cancelled by another message that arrives later during the delivery time, it will be successfully delivered after one time unit.
\end{itemize}

\item[$\bullet$]  {\it CSMA:} The message is admitted to the system only if there is an empty channel; then it will be successfully delivered via one of these channels after one time unit. Otherwise, the message attempt is canceled.
\end{itemize}

In the case of a successful delivery after one time unit, we say that the message has gained access to the medium.

Advantages of pure ALOHA are that it does not need any infrastructure and is therefore cheap to install and run. However, a drawback is that on each arrival of a message might destroy another message that has been already admitted to the system. In turn, this means that each message can be sure to be successfully delivered only one time unit after it has picked an empty channel, namely if has not itself been killed by a later arriving message during its delivery. This means that too large a number of incoming messages (i.e., too large a large value of $\lambda$) decreases on an average the amount of successful deliveries in the system on a long time. We will specify this in terms of a law of large numbers and will see that only a certain percentage of all the channels are typically busy in order to achieve an optimal throughput in this protocol, and we will identify this value.

In CSMA, every admitted message will definitely be successfully delivered after its delivery time. However, in contrast with the ALOHA protocol, some extra information (namely the information about free channels) needs to be constantly provided. Hence, increasing the message density $\lambda$ increases the number of busy channels on an average and hence the throughput (which we quantify below), but also the average number of refused messages (which we neglect in this paper).

\subsection{Our results}\label{sec-Results}

Introduce $A(t)$ as the number of sending attempts in the time interval $[0,t]$ and $S(t)$ as the number of successful transmissions during this time interval. Then $(A(t))_{t\in[0,\infty)}$ is the counting process for the PPP$(\lambda)$, but $(S(t))_{t\in[0,\infty)}$ is highly non-trivial and is the main objective. We formulate our results on the limiting expectation in Section~\ref{sec-expval}, on a crucial Markov chain in Section~\ref{sec-MC} and on the probabilities of large deviations in Section~\ref{sec-LD}.

\subsubsection{Limiting expectation}\label{sec-expval}

Let us calculate the limiting expectation of the number of successfully delivered messages:

\begin{lemma}[Expected throughput]\label{lem-expval}
For both models, $*\in\{{\rm ALOHA, CSMA}\}$, and for any $\lambda\in(0,\infty)$ and $\kappa\in\N$,
\begin{equation}\label{sdef}
s_*(\lambda,\kappa)=\lim_{t\to\infty}\frac 1t S(t)
\end{equation}
exists and is equal to 
\begin{eqnarray}
s_{\rm CSMA}(\lambda, \kappa)&=&\lambda \frac{\sum_{n=0}^{\kappa-1}\frac{\lambda^n}{n!}}{\sum_{n=0}^{\kappa}\frac{\lambda^n}{n!}}=\lambda \frac{\Poi_\lambda([0,\kappa-1])}{\Poi_\lambda([0,\kappa])},\\
s_{\rm ALOHA}(\lambda, \kappa)&=&\lambda \e^{-\frac{\kappa+1}{\kappa}\lambda}\sum_{n=0}^{\kappa-1}\frac{\lambda^n}{n!}\frac{\kappa-n}{\kappa}
=\lambda \e^{-\frac \lambda\kappa}\e^{-\lambda}\Big[\sum_{n=0}^{\kappa-1}\frac{\lambda^n}{n!}-\frac{\lambda}{\kappa}\sum_{n=0}^{\kappa-2}\frac{\lambda^n}{n!}\Big]\\
&=&\lambda \e^{-\frac \lambda\kappa} \Big[\Poi_\lambda([0,\kappa-1])-\frac{\lambda}{\kappa}\Poi_\lambda([0,\kappa-2])\Big].
\end{eqnarray}
\end{lemma}

We wrote $\Poi_\lambda$ for the Poisson-distribution with parameter $\lambda$ on $\N_0$. The proof of Lemma~\ref{lem-expval} is in Section~\ref{sec-expectationALOHA} for the ALOHA case and in Section~\ref{sec-expectationCSMA} for the CSMA case.

In contrast with CSMA, for the ALOHA protocol the question for the optimal value of $\lambda$ for maximizing $s_{\rm ALOHA}(\lambda,\kappa)$ is interesting. An explicit calculation does not seem possible for general $\kappa$. However, using the exponential series approximation for the two sums, we see that
 for $\kappa\to\infty$ the throughput is asymptotically equivalent to
\begin{align*}
\lim_{\lambda,\kappa\to\infty,\frac\lambda\kappa\to x}\frac 1\kappa s_{\rm ALOHA}(\lambda, \kappa)= x(1-x)\e^{-x},\qquad x\in[0,\infty),
\end{align*}
which is easily seen to have a unique maximum for $x=\frac{3-\sqrt{5}}{2}\approx 0.38$ by considering the derivatives. Hence, the optimal throughput is roughly $\sup_\lambda s_{\rm ALOHA}(\lambda,\kappa)\approx 0.38\, \kappa$ for large $\kappa$. Simulations show that for $\kappa=2$ resp. $\kappa=3$ we already have an optimum value for $\lambda\approx 0.43\,\kappa$ resp.~$\lambda\approx 0.41\,\kappa$; the above approximation seems to converge extremely fast.

\subsubsection{A crucial Markov chain}\label{sec-MC}

We are going to introduce now the main object of our ansatz, a certain Markov chain in discrete time that is able to describe the main quantities $A(t)$ and $S(t)$. This Markov chain is not only suitable for describing large-deviation events and their probabilities (see Section ~\ref{sec-LD}), but can generally also be used to derive explicit computer simulations for the entire process of messages and their deliveries.

In both models, we denote by $0<\widetilde T_1<\widetilde T_2<\widetilde T_3<\dots$ all the times at which a message is admitted to a channel. For $i\in\N$ put 
$$
\sigma_i=\widetilde T_i-\widetilde T_{i-1}\qquad\mbox{and}\qquad
A_i=\#\{j\colon T_j\in (\widetilde T_{i-1},\widetilde T_i]\}.
$$
In words, $\sigma_i$ is the length of the time lag between the $(i-1)$st and $i$th admittance of a message to some channel, and $A_i-1$ is the number of refused messages during that time interval. Recall that, for CSMA, $\widetilde T_i$ is the time of the beginning of the $i$-th successful message transmission, but for ALOHA this is only the time of the start of some message transmission attempt; whether or not it will be successful will turn out only at time $\widetilde T_i+1$. Nevertheless, we will show that the sequence $(A_i,\sigma_i)_{i\in\N}$ is suitable to derive precise information about our quantities of interest, $(A(t),S(t))$.

As our first main result, we identify the distribution of the sequence as a kind of Markov renewal process:

\begin{prop}[Markovian structure of $(A_i,\sigma_i)_{i\in\N}$]\label{lem-Markov}
In both models, ALOHA and CSMA, the sequence $(A_i,\sigma_i)_{i\in\N}$ is a $(\kappa-1)$-Markov chain with kernel $W_{\rm CSMA}$ and $W_{\rm ALOHA}$, respectively, from $[\N\times (0,\infty)]^{\kappa-1}$ to $\N\times (0,\infty)$ defined by
\begin{eqnarray}
W_{\rm CSMA}\big((a,t), (k,\d s)\big)&=&\frac{\gamma^{k-1}}{(k-1)!}\lambda^k \e^{-\lambda s}\1_{[\gamma,\infty)}(s)\,\d s,\label{WCSMAdef}\\
W_{\rm ALOHA}\big((a,t), (k,\d s)\big)&=&\frac{\big(\gamma+\frac {B(s)}\kappa\big)^{k-1}}{(k-1)!}\left(1-\frac{\beta(s)}{\kappa}\right)\lambda^k\e^{-\lambda s}\1_{[\gamma,\infty)}(s)\,\d s,\label{WALOHAdef}
\end{eqnarray}
for $(a,t)=\big((a_1,t_1),\cdots,(a_{\kappa-1},t_{\kappa-1})\big)$, where $\gamma:=[1-\sum_{i=1}^{\kappa-1}t_i]_+$, and, for the ALOHA case,
\begin{align*}
\beta(s):=\max\Big\{m\in\N\colon  s+\sum_{j=0}^{m-2}t_{\kappa-1-j}\leq 1\Big\}\wedge\kappa \in\{0,1,\dots,\kappa\}\qquad\mbox{and}\qquad B(s)=\int_0^s\beta (r)\,\d r.
\end{align*}
Both $(\kappa-1)$-Markov chains are uniformly ergodic in the sense that Condition (U) holds (see \eqref{U}).
\end{prop}

The proofs are in Section~\ref{sec-LDPContTimeCSMA} for the CSMA case and in Section~\ref{sec-ALOHAMarkov} for the ALOHA case.

\begin{remark}[Interpretation of the ALOHA kernel]
In the kernel $W_{\rm ALOHA}$, the parameter $\beta(s)$ plays the role of the number of busy channels at time $\widetilde T_i+s$, conditioned on the process of message arrivals before time $\widetilde T_i$. Therefore $1-\frac{\beta(s)}{\kappa}$ is the probability that a message that arrives at that time picks an idle channel. The remaining terms on the right of \eqref{WALOHAdef} express the probability that all the messages arriving during $[\widetilde T_i, \widetilde T_i+\gamma)$ and during $[\widetilde T_i+\gamma,\widetilde T_i+ s)$ pick a busy channel and are therefore not admitted to the system. We will specify this in the proof. 
\hfill$\Diamond$
\end{remark}

\begin{remark}[Markov renewal process] We see that in both cases $(\sigma_i)_{i\in\N}$ is autonomously a $(\kappa-1)$-Markov chain (with a kernel that can easily be deduced from \eqref{WCSMAdef} and \eqref{WALOHAdef}, respectively), and $A_i$ is a random function of $\sigma_{i-1},\sigma_{i-2},\dots,\sigma_{i-\kappa+1}$. More precisely, given the sequence $(\sigma_i)_{i\in\N}$, the variables $A_i$ are independent over $i$ and are Poisson-distributed with a certain parameter depending on $\sigma_{i-1},\sigma_{i-2},\dots,\sigma_{i-\kappa+1}$. Because of this Markovian structure, $(A_i,\sigma_i)_{i\in\N}$ (more precisely, $(A_i,(\sigma_{i-1},\dots,\sigma_{i-\kappa+1}))_{i\in\N}$) is often called a Markov renewal process.
\hfill$\Diamond$
\end{remark}

\begin{remark}[Deriving $S(t)$ and $A(t)$]
In the CSMA case, $t\mapsto S(t)$ is nothing but the time-inverse of the partial sum sequence of the $\sigma_i$ via the formula
\begin{equation}
S(t)=\sup\Big\{m\in\N\colon\sum_{i=1}^m \sigma_i\leq t\Big\}
\end{equation}
and can therefore be fully described in terms of $(\sigma_i)_{i\in\N}$. A similar assertion applies to $A(t)$; see  \eqref{S(t)CSMA}. However, in the ALOHA case one needs additionally $(A_i)_{i\in\N}$ for the description of $S(t)$ (see \eqref{S(t)Aloha}). Certainly, one can also describe other interesting quantities as functionals of the Markov chain, for example the number of unsuccessful messages or (in the ALOHA case) the number of messages that are admitted to some channel, but are canceled later during the service time.
\hfill$\Diamond$
\end{remark}

\subsubsection{Large deviations}\label{sec-LD} Now we turn to our results concerning the large deviations for $(A(t),S(t))$ in the limit $t\to\infty$. Our goal is to quantify the exponential decay rate of the probability of rare events of the form $\{\frac 1t(A(t),S(t))\in B\}$ for many sets $B\subset (0,\infty)^2$.

Let us recall the notion of a {\em large-deviation principle (LDP)}. Indeed, a sequence $(X_n)_{n\in\N}$ of $\Xcal$-valued random variables (where $\Xcal$ is a Polish space) is said to satisfy an LDP with lower semicontinuous rate function $I\colon \Xcal\to[0,\infty]$ if for any closed set $F\subset \Xcal$ and any open set $G\subset\Xcal$,
$$
\limsup_{n \to\infty}\frac 1n\log\P(X_n\in F)\leq -\inf_F I\qquad\mbox{and}\qquad \liminf_{n \to\infty}\frac 1n\log\P(X_n\in G)\geq -\inf_G I.
$$
The first statement is called the large-deviations upper bound (or LDP upper bound), the latter the large-deviations lower bound. These together can be very roughly summarized by saying that $\P(X_n\approx x)\approx \e^{-n I(x)}$ for any $x\in\mathcal X$ as $n\to\infty$. However, topological subtleties are always present in an LDP. See \cite{dembozeitouni} for an account on LDP theory.

Indeed, with the help of the Markov renewal process $(A_i,\sigma_i)_{i\in\N }$ we are in an excellent position to find and prove such an LDP and to identify the rate functions for the two cases; indeed there are very obvious candidates, which are based on the sequence of empirical pair measures of  $(A_i,\sigma_i)_{i\in\N }$. However, there is a problem that we could not overcome, and hence we are only able to derive the LDP upper bound. This problem does not seem to be only technical; it is the fact that $(A(t),S(t))$ is not a {\em continuous} functional of the empirical measure. This makes it impossible to use standard arguments, and we found no way around it; hence we state only the LDP upper bound below. It is not clear to us whether or not the corresponding lower bound holds as well.

The rate functions will be identified in terms of certain entropies, which are well-known in the LDP-theory for Markov chains. Indeed, we write $H(\mu\mid\nu)=\int\d \mu\log\frac{\d\mu}{\d \nu }$ for the relative entropy of a probability measure $\mu$ with respect to another one, $\nu$ (if the density exists, otherwise $H(\mu\mid\nu)=\infty$). Furthermore, for a measure $\mu$ on $\Xcal^\kappa$, we write $\mu^{\ssup{\kappa-1}}$ for the projection of $\mu$ on the vector of the first $\kappa-1$ components, and we say that $\mu$ lies in $\Mcal_1^{\ssup{\rm s}}(\Xcal^\kappa)$ if $\mu$ is a probability measure on $\Xcal^\kappa$ whose projection on the vector of the first $\kappa-1$ components is equal to its projection on the vector of the last $\kappa-1$ components. We abbreviate $\Sigma:=\N\times(0,\infty)$ and define the projections $\pi_1\colon \Sigma^\kappa \to\N$ and $\pi_2\colon \Sigma^\kappa \to (0,\infty)$ by  $\pi_1((a_1,r_1),\dots,(a_\kappa,r_\kappa))=a_\kappa$ and $\pi_2((a_1,r_1),\dots,(a_\kappa,r_\kappa))=r_\kappa$. We write $\langle f, \mu\rangle$ for the integral of a function $f$ with respect to a measure $\mu$.

Our main asymptotic large-deviation results as $t\to\infty$ are as follows.

\begin{theorem}[Large-deviation upper bound for $(A(t),S(t))$]\label{thm-LDP}
In both cases, ALOHA and CSMA, as $t\to\infty$, the pair $\frac 1t (A(t),S(t))$ satisfies an LDP upper bound on $(0,\infty)^2$, i.e.,\ for any closed set $F\subseteq (0,\infty)^2$ we have
\begin{align*}
\limsup_{t\to\infty}\frac{1}{t}\log\mathbb{P}\Big(\frac{1}{t}\big(S(t),A(t)\big)\in F\Big)\leq -\inf_F I_*,
\end{align*}
where, for $a,s\in[0,\infty)$,
\begin{equation*}
I_{\rm CSMA}(a,s)=\sup_{A\in\R,B\in(-\infty,\lambda)}\inf_{\mu\in\Mcal_1^{\ssup{\rm s}}(\Sigma^\kappa)}\big[A\left(a-s\left<\pi_1,\mu\right>\right)+B\left(1-s\left<\pi_2,\mu\right>\right)+H(\mu\mid \mu^{\ssup{\kappa-1}}\otimes W_{\rm CSMA})\big]
\end{equation*} 
for the CSMA protocol and
\begin{equation*}
\begin{aligned}
I_{\rm ALOHA}(a,s)=\sup_{A\in\R,B\in(-\infty,\lambda)}\inf_{\mu\in\Mcal_1^{\ssup{\rm s}}(\Sigma^\kappa)}\Big[A\big(a-\smfrac{a+s}{2}\left<\pi_1,\mu\right>\big)&+B\big(1-\smfrac{a+s}{2}\left<\pi_2,\mu\right>\big)\\
&+H(\mu\mid \mu^{\ssup{\kappa-1}}\otimes W_{\rm ALOHA})\Big]
\end{aligned}
\end{equation*} 
in the case of ALOHA protocol. Both rate functions have precisely one minimizer and are convex on $(0,\infty)^2$ and are good (i.e., their level sets $\{(a,s)\colon I_*(a,s)\leq C\}$ are compact for any $C$). The family $(\frac 1t(A(t),S(t)))_{t>0}$ is exponentially tight (i.e., for any $M>0$ there is a $K>0$ such that $\P(\frac 1t(A(t),S(t))\notin [0,K]^2)\leq \e^{-tM}$ for any $t$).
\end{theorem}

The proof is in Section~\ref{sec-ProofLDCSMA}. By the well-known contraction principle (saying that LDPs are obtained under continuous images and gives an explicit formula for the rate function), we obtain without work:

\begin{cor}[LDP upper bound for number of successfully delivered messages]\label{cor-S(t)LDP}
As $t\to\infty$, $\frac  1tS(t)$ satisfies an LDP upper bound with rate function 
\begin{equation}\label{CSMA_S}
I^S_{\rm CSMA}(s)=\sup_{B\in(-\infty,\lambda)}\inf_{\mu\in\Mcal_1^{\ssup{\rm s}}(\Sigma^\kappa)}\left\{B\left(1-s\left<\pi_2,\mu\right>\right)+H(\mu\mid \mu^{\ssup{\kappa-1}}\otimes W_{\rm CSMA})\right\}
\end{equation}
for the CSMA protocol resp.
\begin{equation}\label{ALOHA_S}
I^S_{\rm ALOHA}(s)=\sup_{B\in(-\infty,\lambda)}\inf_{\mu\in\Mcal_1^{\ssup{\rm s}}(\Sigma^\kappa)}\left\{B\left(1-\frac{\left<\pi_2,\mu\right>s}{2-\left<\pi_1,\mu\right>}\right)+H(\mu\mid \mu^{\ssup{\kappa-1}}\otimes W_{\rm ALOHA})\right\}
\end{equation}
for the ALOHA protocol.
\end{cor}

The proof of Corollary~\ref{cor-S(t)LDP} is immediate from the contraction principle, noting that the canonical projection $\N\times(0,\infty)\to\N$ is continuous.

\begin{remark}[Expected throughput and rate function]
Let us mention that the expected throughput $s_*(\lambda, \kappa)
$ that we identified in Lemma~\ref{lem-expval} can also be characterized in a standard way as the minimizer of the rate function $I^S_*$ for $*\in\{{\rm ALOHA, CSMA}\}$. However, it is rather difficult to identify it from this reasoning since we have no closed formula for the invariant distribution of the Markov chain $(A_i,\sigma_i)_{i\in\N}$, such that we do not have any results in this respect.
\hfill$\Diamond$
\end{remark}

\begin{remark}[Contracting to $A(t)$] Instead of contracting the pair $(A(t),S(t))$ to $S(t)$, we could do this also with $A(t)$ and obtain an analogue of Corollary~\ref{cor-S(t)LDP}. However, since $A(t)$ is nothing but the PPP, one can derive an LDP for $\frac 1t A(t)$ also with much simpler means and obtains the rate function $I^A_{*}(a)=\lambda-a+a\log \frac{a}{\lambda}$.
\hfill$\Diamond$
\end{remark}

\begin{remark}[An application] We think that exponential estimates for the probabilities of rare events like in Corollary~\ref{cor-S(t)LDP} are very relevant for the understanding of the strengths and shortcomes of such telecommunication systems. Indeed, they give us not only extremely good estimates for such probabilities, but also an analytic starting point for getting more information about the most likely situation that governs the rare event: the variational formula for the rate function.

As an example, the probability of the event $\{S(t)\leq t [s_{*}(\lambda,\kappa)-\eps ]\}$ that the number of successfully delivered messages in the time interval $[0,t]$ is not larger than $t [s_{*}(\lambda,\kappa)-\eps]$ is upper bounded by $\exp\{-t \inf_{s\in[0,s_{*}(\lambda,\kappa)-\eps]}I_*^S(s)\}$ as $t\to\infty$, for any $\eps>0$. This means that this probability decays exponentially fast with rate at least $\inf_{[0,s_{*}(\lambda,\kappa)-\eps]}I_*^S=I_*^S(s_{*}(\lambda,\kappa)-\eps)>0$. 
\hfill$\Diamond$
\end{remark}

\begin{remark}[Why not a full LDP?]\label{rem-notfullLDP}
It would be rather desirable to have a full LDP for $(A(t),S(t))$ with rate function $I_*$, but there is an obstacle that we could not overcome: the lack of continuity of the map $\mu\mapsto \langle \pi_2,\mu\rangle$, since $\pi_2$ is unbounded. This makes an application of the crucial G\"artner--Ellis theorem impossible, since it prevents us from proving that $I_*$ is strictly convex, and therefore from proving that its Legendre transform is differentiable. It also prevents us from using the contraction principle, since $(A(t),S(t))$ is not a continuous functional of the empirical measures of $(A_i,\sigma_i)_{i\in\N}$ (see Section~\ref{sec-ProofLDCSMA}). With a lot of more work, we would be able to derive LDP lower bounds that are severely restricted, and the restriction would not be easy to understand, so we abstained from formulating any lower bound.
\hfill$\Diamond$
\end{remark}

\subsection{Related literature} \label{sec-literature}

In \cite{erlang} and \cite{kendall} the CSMA model in continuous time is modeled with the help of a queue that expresses the number of messages that are present in the $\kappa$ channels as a function of the time parameter. Both restricted to the case of exponential distributed service times, which lead to a Markov model, whose invariant distribution can be calculated easily and explicitly. More general results involving arbitrary service time distributions can be found in \cite{borovkov}. Here the process $(Q_n)_{n\in\N}$ of the numbers $Q_n$ of messages in the channels at the time of the arrival of the $n$-th message is considered and its limiting distribution is calculated explicitly depending on the arrival and service time distributions. This and additional ad-hoc methods make it possible to obtain information about the number of successes at a late time, e.g., a law of large numbers. Unfortunately, the process $(Q_n)_{n\in\N}$ does not have the Markov property, but an infinitely long memory. Hence, probabilistic formulas could not be derived, and large deviations for  the throughput of the system could  not be considered. 

In \cite{MAP} one finds the throughput of the single-channel continuous time CSMA, which is $\frac{\lambda}{\lambda+1}$
and coincides with our formula for $s_{\rm CSMA}(\lambda,1)$ in the case $\kappa=1$. Another version of CSMA, namely {\em slotted (single channel) CSMA}, has been studied more intensively than the continuous time model (see \cite{MAP}, \cite{CISS},  \cite{wang} and \cite{lakatos}), and provides also the same limiting throughput as the continuous time model in the single channel case. Let us mention that an analogous large-deviation analysis of multi-channel discrete-time versions of ALOHA and slotted ALOHA and CSMA is carried out in \cite{KK22}.

Since \cite{abramson}, the single-channel pure ALOHA has been studied intensely (for a general overview see \cite{lakatos}, \cite{MAP} and \cite{throughputevaluation}). The throughput is identified there as $\lambda \e^{-2\lambda}$, which also coincides with our result for $s_{\rm ALOHA}(\lambda,1)$ in the special case $\kappa=1$. In \cite{performance}, \cite{lakatos}, \cite{MAP} and \cite{throughputevaluation} one can also read about another, more popular and better known, single channel version of ALOHA, namely the {\em slotted ALOHA}, with the higher throughput $\lambda \e^{-\lambda}$. The multichannel case of this model has also been studied, e.g., in \cite{shen}, where the throughput $\lambda \e^{-\frac{\lambda}{\kappa}}$ has been calculated. See \cite{KK22} for a derivation of this value via a large-deviation analysis with explicit rate functions. To the best of our knowledge, there are no similar results for the multichannel model in continuous time in the literature yet, hence we think that our Lemma~\ref{lem-expval} is novel.

\section{Expectation of the throughput}\label{sec-ExpThrough}

Let us  derive formulas for the expected throughput in the two protocols, i.e., formulas for the expectation of the large-$t$ limit of $\frac 1t S(t)$. In Section~\ref{sec-expectationCSMA} and Section~\ref{sec-expectationALOHA}, respectively, we consider CSMA and ALOHA. This section has nothing to do with the Markov chains introduced in Section~\ref{sec-MC}.

\subsection{CSMA}\label{sec-expectationCSMA}

We borrow some knowledge that was gained in \cite{borovkov}. We pointed out in Section~\ref{sec-literature} that the CSMA system was analysed in \cite{borovkov} with the help of a stochastic process $Q=(Q_n)_{n\in\mathbb{N}}$, where $Q_n$ denotes the number of messages in the $\kappa$ channels (the number of busy channels) at the arrival time of the $n$-th message. The limiting distribution $\nu_Q$ of $Q_n$, as $n$ goes to infinity, has been calculated there as
\begin{align*}
\nu_Q(i)=\frac{\frac{\lambda^i}{i!}}{\sum_{k=0}^{\kappa}\frac{\lambda^k}{k!}}= \Poi_\lambda|_{[0,\kappa]}(i), \qquad i\in\{0,1,\dots,\kappa\},
\end{align*}
where we wrote $\Poi_\lambda|_{[0,\kappa]}$ for the Poisson distribution with parameter $\lambda$, conditioned on being $\leq \kappa$. It was also proved there that the limiting distribution of $Q_n$ as $n\to\infty$ coincides with the limiting distribution of the number of busy channels at a deterministic time $t$ as $t\to\infty$. Using this result, we obtain
\begin{align}\label{formelcsma}
s_{\rm CSMA}(\lambda,\kappa)=\lambda \frac{\sum_{n=0}^{\kappa-1}\frac{\lambda^n}{n!}}{\sum_{n=0}^{\kappa}\frac{\lambda^n}{n!}}=\lambda \frac{\Poi_\lambda([0,\kappa-1])}{\Poi_\lambda([0,\kappa])},
\end{align}
since the average number of the number of successes is equal to the arrival rate $\lambda$, multiplied by the success probability, which is given by the second factor, as every new arriving message can only be delivered successfully, if there are at most $\kappa-1$ busy channels. 

In \eqref{formelcsma} one sees that $s_{\rm CSMA}(\lambda,\kappa)$ is increasing in $\lambda$ and converges to $\kappa$ as $\lambda\to\infty$, which is intuitively clear, because most of the channels are likely to be busy if the arrival rate is high. Therefore, there is no interesting optimisation task over $\lambda$, since the throughput get always better if the density of message is increased. Taking into account also the number of unsuccessful messages (which explodes as $\lambda\to\infty$) makes this issue more interesting, but we do not strive on this here. 

Let us finally mention that in the special case $\kappa=1$ our expression \eqref{formelcsma} yields $s_{\rm CSMA}(\lambda,1)=\frac{\lambda}{\lambda+1}$, which was formerly known (see e.g., \cite{MAP}).

\subsection{ALOHA}\label{sec-expectationALOHA}

Let us calculate the expected limiting throughput in the ALOHA protocol by hand. The expression that we obtain is good enough for also finding the optimal value of the density $\lambda$, at least for large $\kappa$.

Since we are looking at the limit of late times, we need to analyse the ALOHA process in equilibrium. This can be realised by extending the PPP from $[0,\infty)$ to $\R$ and to consider its Palm measure given that one message arrives at time $0$. Write $A([a,b))$ for the number of incoming messages during the time interval $[a,b)$. Fortunately, the number of busy channels at time $0$ depends only on the PPP during the time interval $[-1,0)$, i.e., on $A([-1,0))$. Indeed, the probability of having at least one available channel and taking one of those is equal to
\begin{equation}\label{probincoming}
\sum_{n=0}^{\kappa-1}\mathbb{P}\left(A([-1,0))=n\right)\cdot\left(1-\frac{n}{\kappa}\right)=\sum_{n=0}^{\kappa-1}\frac{\lambda^n}{n!}\e^{-\lambda}\frac{\kappa-n}{\kappa}.
\end{equation}
Note that the PPP has the property that $A([-1,0))$ does not depend on the incoming messages after time $0$; these are the only ones that might influence the transmission success of the message that arrived at time $0$. Hence, for the success probability we have to multiply the term in \eqref{probincoming} with the probability that the message that arrived at time $0$ does not get destroyed afterwards during its service time $(0,1]$, which is equal to
\begin{align*}
\sum_{n=0}^{\infty}\mathbb{P}\left(A((0,1])=n\right)\left(\frac{\kappa-1}{\kappa}\right)^n=&\sum_{n=0}^{\infty}\frac{\lambda^n}{n!}\e^{-\lambda}\left(\frac{\kappa-1}{\kappa}\right)^n=\e^{\lambda\left(\frac{\kappa-1}{\kappa}\right)-\lambda}=\e^{-\frac{\lambda}{\kappa}}.
\end{align*}
Hence, 
\begin{equation}\label{ThroughputALOHA}
s_{\rm ALOHA}(\lambda,\kappa)=\lambda \e^{-\frac{\kappa+1}{\kappa}\lambda}\sum_{n=0}^{\kappa-1}\frac{\lambda^n}{n!}\frac{\kappa-n}{\kappa}
=\lambda \e^{-\frac \lambda\kappa}\e^{-\lambda}\Big[\sum_{n=0}^{\kappa-1}\frac{\lambda^n}{n!}-\frac{\lambda}{\kappa}\sum_{n=0}^{\kappa-2}\frac{\lambda^n}{n!}\Big].
\end{equation}
In the case $\kappa=1$ this is equal to $\lambda \e^{-2\lambda}$, which was already known; see Section~\ref{sec-literature}. This is optimized at $\lambda=\frac 12$ with value $s_{1/2,1}=\frac 1{2\e}\approx \,.18$.

\section{Markov approach}\label{sec-LDPContTime}

\noindent In this section, we introduce suitable Markov chains for both protocols, CSMA and ALOHA, that  are able to describe the number of successful and unsuccessful sending attempts by time $t$. Again, we keep $\lambda\in(0,\infty)$ and $\kappa\in\N$ fixed.

\subsection{Markov approach to  CSMA}\label{sec-LDPContTimeCSMA}

Let us model the CSMA protocol  in terms of a stochastic process in discrete time. Recall that $0<T_1<T_2<T_3<\dots$ denotes all the times at which a message comes in and asks for being admitted to one of the $\kappa$ channels. According to our assumptions, $(T_i)_{i\in\N}$ is a standard Poisson point process (PPP) in $[0,\infty)$ with parameter $\lambda$, and we denote $\tau_i=T_i-T_{i-1}$. It is convenient to introduce the counting process $\Ncal$ defined by $\Ncal(I)=\#\{i\in\N\colon T_i\in I\}$ (the number of sending attempts during the time interval $I$)  for intervals $I$. Then $A(t)=\Ncal([0,t])$ is the number of attempts by time $t$.

By $(\widetilde T_i)_{i\in\N}$ we denote the subsequence $(T_{k(i)})_{i\in\N}$ of $(T_i)_{i\in\N}$ of all those times $T_j$ at which the incoming message is admitted to a channel (i.e., at which not all the $\kappa$ channels are busy); then the delivery takes place during the time interval $[T_{k(i)},T_{k(i)}+1]$, and at time $T_{k(i)}+1$ the message is successfully delivered. We introduce the counting process $\Ncal^{\ssup{\rm s}}$ defined by $\Ncal^{\ssup{\rm s}}(I)=\#\{j\in\N\colon \widetilde T_j\in I\}$ for measurable sets $I\subset[0,\infty)$. Then $S(t)=\Ncal^{\ssup{\rm s}}([0,t])$ is the number of successfully delivered messages during the time interval $[0,t]$. We put $\sigma_i=\widetilde T_i-\widetilde T_{i-1}$, and we register the number 
$$
A_i=\#\{\text{attempts during}\; (\widetilde T_{i-1},\widetilde T_{i}]\}=\sum_{j\in\N}\1\{T_j\in(\widetilde T_{i-1},\widetilde T_{i}]\}=\Ncal\big((\widetilde T_{i-1},\widetilde T_{i}]\big)\in\N
$$
of sending attempts in the time interval $(\widetilde T_{i-1},\widetilde T_{i}]$. Then we have, for any $k\in\N$ and $t\in[0,\infty)$,
\begin{equation}\label{Description}
\mbox{on the event }\{\widetilde T_k\leq t<\widetilde T_{k+1}\},\qquad S(t)=k\quad\mbox{and}\quad\sum_{i=1}^k A_i\leq A(t)<\sum_{i=1}^{k+1} A_i.
\end{equation}
Hence, we are able to express the main quantities, $A(t)$ and $S(t)$, in terms of the sequence $(A_i,\sigma_i)_{i\in\N}$ whose state space is equal to 
$$
\Sigma=\N\times (0,\infty).
$$ 
Therefore, we want to describe its distribution. It turns out that it is in general not a Markov chain, but a $(\kappa-1)$-Markov chain, i.e., a stochastic process with a memory of length $\leq \kappa-1 $:

\begin{lemma}[Markovian structure of $(A_i,\sigma_i)_{i\in\N}$ in CSMA case]\label{lem-MarkovCSMA}
The sequence $(A_i,\sigma_i)_{i\in\N}$ is a  time-homogeneous $(\kappa-1)$-Markov chain with kernel $W_{\rm CSMA}$ from $\Sigma^{\kappa-1}$ to $\Sigma$ defined by
\begin{equation*}
W_{\rm CSMA}\big(\big((a_1,t_1),\cdots,(a_{\kappa-1},t_{\kappa-1})\big), (k,\d s)\big)=\frac{(\lambda\gamma)^{k-1}}{(k-1)!}\lambda \e^{-\lambda s}\1_{[\gamma,\infty)}(s)\,\d s,
\end{equation*}
where $\gamma=[1-\sum_{k=1}^{\kappa-1}t_k]_+$.
\end{lemma}

\begin{proof}
Let us fix $i\in\N$ and identify the conditional distribution of $(A_{i+1},\sigma_{i+1})$ given $(A_j,\sigma_j)_{j\leq i}$. This will turn out to be the same as the conditional distribution given the $(\kappa-1)$-past $(A_j,\sigma_j)_{j\in\{i-\kappa+2,\dots,i\}}$, and $W_{\rm CSMA}$ will turn out to be a version of this conditional distribution; this will finish the proof of the lemma.

Conditioning on $(A_j,\sigma_j)_{j\leq i}$ includes conditioning on $(\widetilde T_j)_{j\leq i}=(T_{k(j)})_{j\leq i}$. The next arrival time $T_j$ with an idle channel after time $\widetilde T_i$ is the first $T_j$ after $\widetilde T_i$ such that no more than $\kappa-1$ messages are in the $\kappa$ channels at this time. Since all the messages that are currently in the system have arrived in the last time unit before, we can say that this next $T_j$ is the first $T_j$ after $\widetilde T_i$ such that in the time interval $(T_j-1,T_j)$ the number of the $\widetilde T_k$ is smaller than $\kappa$. In formulas,
$$
\begin{aligned}
\widetilde T_{i+1}=T_{k(i+1)}&=\inf\{T_j\colon j>k(i), \Ncal^{\ssup{\rm s}}((T_j-1,T_j))<\kappa\}\\
&=\inf\{T_j\colon j >k(i), T_j-\widetilde T_{i-\kappa+1}>1\}.
\end{aligned}
$$
In terms of the time differences, we see that $\sigma_{i+1}=\tau_{k(i)+1}+\tau_{k(i)+2}+\dots+\tau_{k(i+1)}$ with 
$$
k(i+1)=\inf\Big\{j>k(i)\colon \tau_{k(i)+1}+\tau_{k(i)+2}+\dots+\tau_{j} >\gamma_i\Big\}, \quad\mbox{where }\gamma_i=\Big[1-\sum_{k=0}^{\kappa-2}\sigma_{i-k}\Big]_+.
$$

In other words, given $(\sigma_j)_{j\leq i}$, the conditional distribution of $\sigma_{i+1}$ is equal to the first point of a PPP$(\lambda)$ after time $\gamma_i$. Using the well-known memoryless property of the PPP, we see that this distribution is the distribution of $\gamma_i+X$, where $X$ is an independent Exp$(\lambda)$-distributed random variable. This distribution has the density $s\mapsto \lambda \e^{-(s-\gamma_i)\lambda}\1_{[\gamma_i,\infty)}(s)$.

We also see that $A_{i+1}-1$, the number of unsuccessful sending attempts in the time interval $(\widetilde T_i,\widetilde T_{i+1}]$, has the conditional distribution equal to the one of  $\Ncal((0, \gamma_i])$, the Poisson distribution with parameter $\lambda \gamma_i$. Hence, the conditional probability that $A_{i+1}=k$ is equal to $\Poi_{\lambda \gamma_i}(k-1)=\e^{-\lambda\gamma_i}\frac{(\lambda\gamma_i)^{k-1}}{(k-1)!}$.

Summarizing, we see that $(A_{i+1}, \sigma_{i+1})$ given $(A_j,\sigma_j)_{j\leq i}$ depends only on the $(\kappa-1)$-past, i.e., on $(A_j,\sigma_j)_{j\in\{i-\kappa+2,\dots,i\}}$ and that
\begin{equation}
\begin{aligned}
\mathbb{P}&\big((A_{i+1},\sigma_{i+1})\in\d(k, s)\mid (A_i, \sigma_i)=(a_1,t_1),\cdots,(A_{i-\kappa+2}, \sigma_{i-\kappa+2})=(a_{\kappa-1},t_{\kappa-1})\big)
\\
&=\frac{(\lambda\gamma)^{k-1}}{(k-1)!}\lambda \e^{-\lambda s}\1_{[\gamma,\infty)}(s)\,\d s,
\end{aligned}
\end{equation}
where $\gamma=[1-\sum_{k=1}^{\kappa-1}t_k]_+$. Hence, $(A_i,\sigma_i)_{i\in\N}$ is a $(\kappa-1)$-Markov chain with the kernel given in \eqref{WCSMAdef}.
\end{proof}

The fact that $(A_i,\sigma_i)_{i\in\N}$ is a $(\kappa-1)$-Markov chain can obviously also be rephrased in terms of the sequence of subsequent $(\kappa-1)$-vectors:

\begin{cor}\label{cor-MCCSMA}
Equivalently, one can formulate Lemma~\ref{lem-MarkovCSMA} by saying that the vectors 
\begin{equation}\label{Rdef}
R_i^{\ssup{\rm CSMA}}:=\big((A_i,\sigma_i),\cdots, (A_{i+\kappa-2}, \sigma_{i+\kappa-2})\big)\in \Sigma^{\kappa-1}
\end{equation}
form a time-homogeneous Markov chain $(R_i^{\ssup{\rm CSMA}})_{i\in\mathbb{N}}$ on the state space $\Sigma^{\kappa-1}$ with the transition kernel $P_{\rm CSMA}$ defined by
\begin{equation}\label{PCSMAdef}
\begin{aligned}
P_{\rm CSMA}\big(\big((&a_1,t_1),\cdots,(a_{\kappa-1},t_{\kappa-1})\big), \d\big((b_1,s_1),\cdots,(b_{\kappa-1},s_{\kappa-1})\big)\big)
\\
&=\bigotimes_{i=1}^{\kappa-2}\delta_{(a_{i+1},t_{i+1})}(\d(b_i,s_i))
\,\otimes W_{\rm CSMA}\big(\big((a_1,t_1),\cdots,(a_{\kappa-1},t_{\kappa-1})\big), \d(b_{\kappa-1},s_{\kappa-1})\big).
\end{aligned}
\end{equation}
\end{cor}

In Section~\ref{sec-ProofLDCSMA} we will need the following strong ergodicity property of the Markov chain $(R_i^{\ssup{\rm CSMA}})_{i\in\N}$. By $P^i$ we denote the $i$-th power of $P$ (in the sense of \lq matrix\rq\ multiplication), i.e., the $i$-step transition kernel, for $i\in\N$.

\medskip\noindent{\bf Condition (U).}{\em We say, a Markov chain in a Polish space $\Sigma$ with transition kernel $P$ satisfies (U) if there exist $\ell,N\in\N$ satisfying $\ell\leq N$ and a constant $M\in[1,\infty)$ such that}
\begin{equation}\label{U}
P^\ell(\sigma,\cdot)\leq \frac MN\sum_{i=1}^N P^i(\tau,\cdot),\qquad\sigma,\tau\in\Sigma.
\end{equation}
\medskip
Condition (U) is a kind of uniform ergodicity property. It implies that the Markov chain has a unique invariant measure \cite[Ex.~6.5.8]{dembozeitouni}.

\begin{lemma}[Uniform ergodicity of $(R_i^{\ssup{\rm CSMA}})_{i\in\N}$]\label{lem-UCSMA}
For any $\lambda\in(0,\infty)$ and $\kappa\in\N$, the Markov chain $(R_i^{\ssup{\rm CSMA}})_{i\in\N}$ introduced in  Corollary~\ref{cor-MCCSMA} satisfies (U).
\end{lemma}

\begin{proof} Instead of the transition kernel $P_{\rm CSMA}$, it will be sufficient to work with the kernel $W_{\rm CSMA}$. We write $W_{\rm CSMA}^{\ssup{i}}$ for the $i$-th power of the kernel $W_{\rm CSMA}$. We will show the existence  of a constant $\widetilde M$ such that
\begin{align}\label{Ugoal}
W_{\rm CSMA}^{\ssup{\kappa+1}}\big((a,t),(k,\d s)\big)/ \d s
\leq \widetilde M\, W_{\rm CSMA}^{\ssup{\kappa+1}}\big((\widetilde a,\widetilde t),(k,\d s)\big)/ \d s
\end{align}
for $a,\widetilde a\in \N^{\kappa-1},t,\widetilde t\in(0,\infty)^{\kappa-1},k\in\N,s\in(0,\infty)$. It is clear that (U) follows from that assertion with $\ell=N=\kappa+1$ and $M=\widetilde M(\kappa+1)$.

From \eqref{WCSMAdef} we see that actually both sides of \eqref{Ugoal}  do not depend on $a$ nor on $\widetilde a$. We write both sides of \eqref{Ugoal} in terms of random variables, more precisely in terms of a $(\kappa-1)$-Markov chain $(\sigma_i)_{i\in\{-\kappa+2,-\kappa+3,\dots\}}$ using the notation $\E_t(\cdot)=\E(\cdot|\sigma_{-\kappa+i+1}=t_i\forall i\in[\kappa-1])$, and then we have
$$
W_{\rm CSMA}^{\ssup{\kappa+1}}\big((a,t),(k,\d s)\big)/ \d s=\frac{\lambda^{k}}{(k-1)!}\e^{-\lambda s}
\E_{t}\big[\gamma^{k-1}\1_{[0,s]}(\gamma)\big],\qquad\mbox{where }\gamma=\Big(1-\sum_{i=2}^\kappa \sigma_i\Big)_+.
$$
Hence, \eqref{Ugoal} is equivalent to 
\begin{equation}\label{Ugoalexpect}
\E_{t}\big[\gamma^{k-1}\1_{[0,s]}(\gamma)\big] \leq M \E_{\widetilde t}\big[\gamma^{k-1}\1_{[0,s]}(\gamma)\big],\qquad t,\widetilde t\in(0,\infty)^{\kappa-1}, k\in\N, s\in(0,\infty).
\end{equation}
We are going to find a lower bound for the expectation on the right by restricting to the event $\{\sigma_1>1\}$, on which $\sigma_2,\dots,\sigma_{\kappa}$ are independent ${\rm Exp}_\lambda$-distributed variables (this reflects the fact that, if for more than one time unit no new message arrives, then all channels are empty and the next $\kappa$ incoming messages will find a free channel). Furthermore, we will derive an upper bound for the left-hand side in terms of a multiple integral involving such random variables. 

Note that
$$
\E_{t}\big[\gamma^{k-1}\1_{[0,s]}(\gamma)\big]
=\int_{(1-s)_+}^1 \d x\, (1-x)^{k-1} \P_t\Big(\sum_{i=2}^\kappa\sigma_i \in \d x\Big)\big/\d x.
$$
Introduce the probability density $f_\gamma(s)= \lambda \e^{-\lambda (s-\gamma)}\1_{[\gamma,\infty)}(s)$ on $[0,\infty)$. Then  a density of $\sum_{i=2}^\kappa\sigma_i$ under $\E_{(t_2,t_3,\dots,t_{\kappa-1},s_1)}$ is the map
$$
(0,\infty)\ni x\mapsto \int_{(0,\infty)^{\kappa-1}}\d s_2\dots \d s_{\kappa}\, \1\Big\{\sum_{i=2}^\kappa s_i=x\Big\}\prod_{i=1}^{\kappa-1} f_{\gamma_i}(s_i), \qquad\mbox{with }\gamma_i=\Big(1-\sum_{j=-\kappa+2+i}^{i-1}s_j\Big)_+,
$$
where we put $s_{-\kappa+j}=t_j$. Now we estimate $f_\gamma(s)\leq \e^\lambda f_0(s)$, then the above density is upper-bounded by $\e^{\lambda (\kappa-1)} f_0^{\star(\kappa-1)}(x)$, where $f_0^{\star(\kappa-1)}$ is the $(\kappa-1)$-fold convolution of $f_0$. This implies that 
$$
\E_{t}\big[\gamma^{k-1}\1_{[0,s]}(\gamma)\big]
\leq \e^{\lambda (\kappa-1)} \int_{(1-s)_+}^1 \d x\, (1-x)^{k-1}f_0^{\star(\kappa-1)}(x).
$$
On the other hand, we may estimate
$$
\begin{aligned}
\E_{\widetilde t}\big[\gamma^{k-1}\1_{[0,s]}(\gamma)\big]
&\geq \E_{\widetilde t}\big[\1\{\sigma_1>1\}\gamma^{k-1}\1_{[0,s]}(\gamma)\big]\\
&=\P_{\widetilde t}(\sigma_1>1) \int_{(1-s)_+}^1 \d x\, (1-x)^{k-1} \P_t\Big(\sum_{i=2}^\kappa\sigma_i \in \d x\Big)\big/\d x\\
&\geq \e^{-\lambda } \int_{(1-s)_+}^1 \d x\, (1-x)^{k-1} f_0^{\star(\kappa-1)}(x).
\end{aligned}
$$
The last two displays together imply our goal, \eqref{Ugoalexpect}, and hence \eqref{Ugoal}.
\end{proof}

\subsection{Markov approach for the ALOHA protocol}\label{sec-ALOHAMarkov}

Now we turn to a similar treatment of the ALOHA protocol. We adopt all the notation from Section \ref{sec-LDPContTimeCSMA}; that is, we fix $\lambda\in(0,\infty)$ and $\kappa \in\N$ and assume that $(T_i)_{i\in\N}$ is a standard PPP($\lambda$) (the sequence of times at which a message comes in and requires a channel for being transmitted) and $\Ncal((a,b])$ is the number of Poisson points in the time interval $(a,b]$ for any $a<b$. 

Recall that, in the ALOHA protocol, each incoming message jumps into a randomly picked one of the $\kappa$ channels, regardless whether it is idle or busy. If it is busy, then it destroys the message that is currently in the channel and itself as well. As a consequence, the new message is rejected immediately, i.e., it does not get access to the system, while the old one first remains in the channel until its service time is over and leaves after one time unit without having been successfully delivered. However, if the channel is idle, then the new message is only \emph{potentially} successful, since it can still be destroyed during the service time by a new arriving one that picks this channel. This uncertain situation remains until one time unit after the entry into the channel; then the message is successfully delivered if it has not been cancelled by then.

We consider the sequence $(\widetilde T_i)_{i\in\N}$ of all the times at which an incoming message picks an idle  channel, a subsequence of $(T_i)_{i\in\N}$. We again put $\sigma_i=\widetilde{T}_i-\widetilde{T}_{i-1}$ and $A_i=A((\widetilde{T}_{i-1},\widetilde{T}_i])$, which is $1+$ the number of incoming messages that jump into some busy channel and therefore destroy the message therein. Recall that $\Sigma=\N\times (0,\infty)$.

\begin{lemma}[Markovian structure of $(A_i,\sigma_i)_{i\in\N}$ in ALOHA case]\label{lem-MarkovALOHA}
The sequence $(A_i,\sigma_i)_{i\in\N}$ is a $(\kappa-1)$-Markov chain with kernel $W_{\rm ALOHA}$ from $\Sigma^{\kappa-1}$ to $\Sigma$ defined by
\begin{equation*}
W_{\rm ALOHA}\big(\big((a_1,t_1),\cdots,(a_{\kappa-1},t_{\kappa-1})\big), (k,\d s)\big)=\frac{\big(\gamma+\frac {B(s)}\kappa\big)^{k-1}}{(k-1)!}\left(1-\frac{\beta(s)}{\kappa}\right)\lambda^k \e^{-\lambda s}\1_{[\gamma_i,\infty)}(s)\,\d s,
\end{equation*}
where we wrote $\gamma:=[1-\sum_{k=1}^{\kappa-1}t_k]_+$ and
\begin{equation}\label{betandef}
\begin{aligned}
\beta(s):=\max\Big\{m\in\N\colon  s+\sum_{j=0}^{m-2}t_{\kappa-1-j}\leq 1\Big\}\wedge\kappa \in\{0,1,\dots,\kappa\} \qquad\mbox{and}\qquad B(s)=\int_0^s\beta(r)\,\d r
\end{aligned}
\end{equation}
for the number of busy channels $s$ time unit after the last successful arrival and its primitive.
\end{lemma}

\begin{proof} We keep $i\in\N$ fixed, condition on $(A_j,\sigma_j)_{j\leq i}$ and examine the distribution of $(A_{i+1},\sigma_{i+1})$. Let us first examine the density of the probability of the event $\{\sigma_{i+1}=s\}$. At time $\widetilde{T}_{i+1}$ there is at least one free channel in order for the arriving message to access the system. Hence, $\widetilde T_{i+1}$ must be after time $\widetilde T_i+\gamma_i$ , where $\gamma_i:=[1-\sum_{n=i-\kappa+2}^{i}\sigma_{n}]_+$, like in the CSMA model, that is, $\sigma_{i+1}>\gamma_i$. However, this time $\widetilde{T}_{i+1}$ is not necessarily the first point of the PPP after $\widetilde T_i+\gamma_i$, but the first Poisson point after $\widetilde T_i+\gamma_i$ at  which an idle channel is picked. Hence, we have to calculate the probability of picking a free channel at an Poisson time point. For this, we need  to know the number of free channels at any arbitrary time after $\widetilde T_i+\gamma_i$.

So let $\widetilde{T}_i+s$, $s>\gamma_i$, be this arbitrary time. If $s>1$, at least one time unit has passed without new incoming messages after $\widetilde{T}_{i}$, which means that all channels are idle again at $\widetilde{T}_i+s$. If $s\leq 1$, at least one channel is busy at time $\widetilde T_i+s$, as there is at least one message, namely the one arrived at $\widetilde{T}_{i}$, whose service time is not over yet. Of course, there could be more messages still remaining in the system, depending on $s$, and we have to determine this relation. It is clear, that if additionally  $s+\sigma_i\leq 1$, then the messages arrived at $\widetilde{T}_{i-1}$ is also still in the system, so there are at least two busy channels at time $\widetilde T_i+s$. Analogously, there must be at least 3 occupied channels, if $s+\sigma_i+\sigma_{i-1}\leq 1$ additionally, as the service time of the message arrived at $\widetilde{T}_{i-2}$ is also not over yet. We see step by step, that if $s+\sum_{k=0}^{\kappa-3}\sigma_{i-k}\leq 1$, we have at least $\kappa-1$ busy channels and if $s+\sum_{k=0}^{\kappa-2}\sigma_{i-k}\leq 1$, all $k$ channels must be busy, as the delivery of all the last $k$ messages is still remaining. Hence, the number of busy channels at time  $\widetilde{T}_i+s$ is given, for $s>\gamma_i$, by
\begin{equation}
\beta(s):=\1_{\{s\leq 1\}}+\1_{\{s+\sigma_i\leq 1\}}+\cdots+\1_{\{s+\sum_{k=0}^{\kappa-2}\sigma_{i-k}\leq 1\}}
=\max\Big\{m\in\N\colon s+\sum_{k=0}^{m-2}\sigma_{i-k}\leq 1\Big\}\wedge\kappa.
\end{equation}
Then it is clear that the probability of picking randomly a busy respectively free channel at time $\widetilde T_i+s$ is equal to $\frac{\beta(s)}{\kappa}$ respectively $1-\frac{\beta(s)}{\kappa}$. So the first point of the PPP after $\widetilde T_i+\gamma_i$ coincides only with probability $1-\frac{\beta(s)}{\kappa}$ with $\widetilde{T}_{i+1}$ (on the event $\{\sigma_{i+1}=s\}$). This yields the density
\begin{align*}
s\mapsto \Big(1-\frac{\beta(s)}{\kappa}\Big)\lambda \e^{-\lambda(s-\gamma_i)}\1_{[\gamma_i,\infty)}(s)
\end{align*} 
for the probability that the first message after time $\widetilde T_i+\gamma_i$ picks a free channel. 

Now we consider, for any $n\in\N$, the event that $\widetilde{T}_{i+1}=\widetilde T_i+s$ is the $n$-th point of the PPP after $\widetilde{T}_{i}+\gamma_i$ at which for the first time an idle channel is picked. On the event that there are precisely $n-1$ Poisson points in the interval $(\widetilde T_i+\gamma_i,\widetilde T_{i+1})$ and another one at $\widetilde{T}_{i+1}=\widetilde T_i+s$, the density of these $n$ Poisson points is equal to
$$
(s_1,\dots,s_{n-1},s)\mapsto \1_{\{\gamma_i<s_1<s_2<\dots<s_{n-1}<s\}} \lambda^{n}\e^{-\lambda(s-\gamma_i)}\,\d s_1\cdots \d s_{n-1}\d s.
$$
On this event, the probability that the first $n-1$ of them pick a busy channel and the last one an idle one is equal to 
\begin{align*}
\Big(1-\frac{\beta(s)}\kappa\Big)\prod_{k=1}^{n-1}\frac{\beta(s_k)}\kappa.
\end{align*}
In order to obtain the density of $\sigma_{i+1}$, we need to integrate over all these $s_1,\dots,s_{n-1}$ and have to sum on $n\in\N$. Hence, the conditional distribution of $\sigma_{i+1}$ is given as
\begin{align*}
\mathbb{P}&(\sigma_{i+1}\in\d s\mid (A_j, \sigma_j)_{j\leq i})\\
&=\sum_{n=1}^{\infty} \Big[\int_{\gamma_i<s_1<s_2<\dots<s_{n-1}<s}\Big(\prod_{k=1}^{n-1}\frac{\beta(s_k)}\kappa\Big)\,\d s_1\d s_2\dots,\d s_{n-1}\Big]\,\Big(1-\frac{\beta(s)}\kappa\Big) \lambda^{n}\e^{-\lambda(s-\gamma_i)}\,\d s\\
&=\sum_{n=1}^{\infty}\frac1{(n-1)!}\Big[\int_0^s \frac{\beta(r)}\kappa\,\d r\Big]^{n-1} \left(1-\frac{\beta(s)}{\kappa}\right)\lambda^n \e^{-\lambda(s-\gamma_i)}\1_{[\gamma_i,\infty)}(s)\,\d s\\
&=\e^{\lambda B(s)/\kappa}\left(1-\frac{\beta(s)}{\kappa}\right)\lambda \e^{-\lambda(s-\gamma_i)}\1_{[\gamma_i,\infty)}(s)\,\d s,
\end{align*}
where we used the exponential series and remind on \eqref{betandef}, now with $\gamma_i$ instead of $\gamma$ (observe that $\beta(s)=0$ on $[0,\gamma_i]$).

Now, we look at the intersection of $\{\sigma_{i+1}=s\}$ with the event $\left\{A_{i+1}=k\right\}$ for $k\in\mathbb{N}$. Here we have $k-1$ unsuccessful attempts during $(\widetilde{T}_{i},\widetilde{T}_{i+1}]$; indeed these are all the Poisson points in the interval $(\widetilde{T}_i, \widetilde{T}_i+\gamma_i]$ plus the ones in the interval $(\widetilde{T}_i+\gamma_i,\widetilde T_{i+1})$ that failed to pick an idle channel, and these two numbers are independent by the properties of the PPP. The number of the first ones have a Poisson distribution with parameter $\lambda \gamma_i$, and the one of the latter ones has been examined above. Hence, the conditional distribution of $A_{i+1}$ is the convolution of these two:
$$
\begin{aligned}
\mathbb{P}&\big(A_{i+1}=k, \sigma_{i+1}\in\d s\mid (A_j, \sigma_j)_{j\leq i}\big)
\\
&=\sum_{n=1}^{k} \frac{(\lambda\gamma_i)^{k-n}}{(k-n)!}\e^{-\lambda\gamma_i}\left(1-\frac{\beta(s)}{\kappa}\right)\frac{B(s)^{n-1}/(n-1)!}{\kappa^{n-1}}\lambda^n \e^{-\lambda(s-\gamma_i)}\1_{[\gamma_i,\infty)}(s)\,\d s
\\
&=\gamma_i^{k-1}\sum_{n=1}^{k} \frac{(B(s)/\gamma_i\kappa)^{n-1}}{(k-n)!(n-1)!}\left(1-\frac{\beta(s)}{\kappa}\right)\lambda^k \e^{-\lambda s}\1_{[\gamma_i,\infty)}(s)\,\d s
\\
&=\frac{\gamma_i^{k-1}}{(k-1)!}\Big(1+\frac{B(s)}{\gamma_i \kappa}\Big)^{k-1}\left(1-\frac{\beta(s)}{\kappa}\right)\lambda^k \e^{-\lambda s}\1_{[\gamma_i,\infty)}(s)\,\d s\\
&=W_{\rm ALOHA}\big(\big((a_1,t_1),\cdots,(a_{\kappa-1},t_{\kappa-1})\big), (k,\d s)\big)\big),
\end{aligned}
$$
where we used the binomial theorem. In particular, we see that we have again a $(\kappa-1)$-Markov chain, as the transition probaility depends only on $\sigma_i,\dots,\sigma_{i-\kappa+2}$ (and by the way, not at all on the $A_j$'s). 
\end{proof}

Analogously to the CSMA case in Section~\ref{sec-LDPContTimeCSMA}, the sequence of $(\kappa-1)$-vectors $R_i^{\ssup{\rm ALOHA}}$ of $(A_i,\sigma_i)_{i\in\N}$ defined as in \eqref{Rdef} form a Markov chain on the state space $\Sigma^{\kappa-1}$ with a transition kernel $P_{\rm ALOHA}$ that is defined analogously to \eqref{PCSMAdef}. Also the analogue to Lemma~\ref{lem-UCSMA} holds:

\begin{lemma}[Uniform ergodicity of $(R_i^{\ssup{\rm ALOHA}})_{i\in\N}$]\label{lem-UALOHA}
For the ALOHA protocol, for any $\lambda\in(0,\infty)$ and $\kappa\in\N$, the Markov chain $(R_i^{\ssup{\rm ALOHA}})_{i\in\N}$  satisfies (U).
\end{lemma}

\begin{proof}
We use the same strategy as in the proof of Lemma \ref{lem-UCSMA} and will prove that \eqref{Ugoal} holds for some $\widetilde M$.  Again abbreviate $\E_t(\cdot)=\E(\cdot|\sigma_{-\kappa+i+1}=t_i\forall i\in[\kappa-1])$ for $t=(t_1,\cdots,t_{\kappa-1})\in (0,\infty)^{\kappa-1}$. Then, using the notation $\gamma=\left[1-\sum_{i=2}^{\kappa}\sigma_i\right]_+$ and
\begin{align*}
\beta^{\ssup\sigma}(s):=\1_{[\gamma,\infty)}(s)\,\max\Big\{m\colon s+\sum_{i=0}^{m-2}\sigma_{\kappa-i}\leq 1\Big\}\wedge\kappa
\end{align*}
we can write
\begin{align*}
W_{\rm ALOHA}^{\ssup{\kappa+1}}\big((a,t),(k,\d s)\big)/ \d s=\frac{\lambda^k}{(k-1)!}\left(1-\frac{\beta(s)}{k}\right)e^{-\lambda s} \mathbb{E}_t\left[\left(\gamma+\frac{\beta^{\ssup\sigma}(s)}{\kappa}\right)^{k-1}  \1_{[0,s]}(\gamma)\right];
\end{align*}
observe that  $\beta^{\ssup\sigma}(s)$ and $\gamma$ are also functions of the random variables $\sigma_2,\cdots,\sigma_{\kappa}$. 

For our goal it is sufficient to show the existence of some $\widetilde M>0$ such that, for each $t,\tilde{t}\in(0,\infty)^{\kappa-1}$, $s>0$, $k\in\mathbb{N}$ and $n\leq k$
\begin{align*}
\mathbb{E}_t\big[G(\sigma_2,\dots,\sigma_{\kappa})\big]\leq \widetilde M \mathbb{E}_{\tilde{t}}\big[G(\sigma_2,\dots,\sigma_{\kappa})\big],\qquad\mbox{for }G(\sigma_2,\dots,\sigma_{\kappa})=\left(\gamma+\frac{\beta^{\ssup\sigma}(s)}{\kappa}\right)^{k-1}  \1_{[0,s]}(\gamma).
\end{align*}
We will show this even for any non-negative measurable function $G$ (with the same constant $\widetilde M$) by showing the corresponding inequality for the respective densities of $(\sigma_2,\dots,\sigma_{\kappa})$ under $\P_t$ and $\P_{\widetilde t}$.

Recall the probability density $f_{g}(s)=\lambda \e^{-\lambda(s-g)}\1_{[g,\infty)}(s)$ on $s\in[0,\infty)$ for any $g\in[0,\infty)$. Now, a density of $(\sigma_2,\dots,\sigma_{\kappa})$ under $\mathbb{E}_{(t_2,t_3,\cdots,t_{\kappa-1},s_1)}$ is the map
\begin{equation}\label{density}
(s_2,\dots,s_\kappa)\mapsto \prod_{i=1}^{\kappa-1}f_{\gamma_i}(s_i),
\end{equation}
where we wrote $s_{-\kappa+j}=t_j$ and $\gamma_i:=[1-\sum_{j=-\kappa+2+i}^{i-1}s_j]_+$. Since $f_g(s)\leq \e^{\lambda} f_0(s)$ for any $g,s\in[0,\infty)$, this density is upper bounded by $\e^{\lambda(\kappa-1)}f_0^{\otimes(\kappa-1)}(s_2,\dots,s_\kappa)$, where $f_0^{\otimes(\kappa-1)}$ is the $(\kappa-1)$-fold tensor product of $f_0$. Then we can upper bound the left-hand side as follows (the expectation after the first equality is on $\sigma_1$):
\begin{align*}
\mathbb{E}_t\big[G(\sigma_2,\dots,\sigma_{\kappa})\big]
&=\E_t\Big[\int_{[0,\infty)^{\kappa-1}} G(s_2,\dots,s_\kappa)\,\prod_{i=1}^{\kappa-1}f_{\gamma_i}(s_i)\Big]\\
&\leq \e^{\lambda(\kappa-1)}\E_t\Big[\int_{[0,\infty)^{\kappa-1}} G(s_2,\dots,s_\kappa)\,f_0^{\otimes(\kappa-1)}(s_2,\dots,s_\kappa)\Big]\\
&=\e^{\lambda(\kappa-1)}\int_{[0,\infty)^{\kappa-1}} G(s_2,\dots,s_\kappa)\,f_0^{\otimes(\kappa-1)}(s_2,\dots,s_\kappa).
\end{align*}

On the other hand, if we restrict to the event $\{\sigma_1>1\}$, then $\sigma_2,\dots,\sigma_{\kappa}$ are again independent exponentially distributed variables and independent of $\sigma_1$, and we can estimate
\begin{align*}
\mathbb{E}_{\widetilde{t}}\big[G(\sigma_2,\dots,\sigma_{\kappa})\big]
&\geq \mathbb{E}_{\widetilde{t}}\big[\1_{\{\sigma_1>1\}}G(\sigma_2,\dots,\sigma_{\kappa})\big]
\\
&=\mathbb{P}_{\widetilde{t}}(\sigma_1>1)
\int_{[0,\infty)^{\kappa-1}} G(s_2,\dots,s_\kappa)\, f_0^{\otimes(\kappa-1)}(s_2,\dots,s_\kappa)\,\d s_2,\dots,\d s_\kappa.
\end{align*}
Since $\mathbb{P}_{\widetilde{t}}(\sigma_1>1)\geq \e^{-\lambda}$, this implies the assertion with $\widetilde M=\e^{\lambda\kappa}$.
\end{proof}

\section{Large deviations}
\label{sec-ProofLDCSMA}
\noindent  In this section, we prove a large deviation upper bound for the pair $\frac{1}{t}(S(t),A(t))$ for both protocols. We can make most of the steps jointly for both protocols. The basis of our large-deviation analysis is the empirical pair measures and the empirical $\kappa$-string measures $L_n^\kappa$ of the Markov chain that we introduced in Section~\ref{sec-MC}. The two LDPs for $(L_n^\kappa)_{n\in\N}$ as $n\to\infty$ are easily derived from general theory, and the main object, $(A(t),S(t))$ is a kind of time-inverse of $n\mapsto (\langle \pi_1,L_n^\kappa\rangle, \langle \pi_1,L_n^\kappa\rangle)$. However, there are two problems left: The latter is {\em a priori} not a continuous functional of $L_n^\kappa$, and we need to make the step from an LDP for this pair to the pair $(A(t),S(t))$. These two major steps will be done in Lemmas~\ref{lem-LDPone} and \ref{LDPCSMA}. However, we were not able to overcome the lack of continuity of $\mu\mapsto \langle \pi_2,\mu\rangle$ and cannot derive a full LDP for $(A(t),S(t))$.

Let us abbreviate $\Sigma=\N\times (0,\infty)$ and let $*\in\{\rm CSMA, \rm ALOHA\}$. We introduce the {\em empirical pair measure} of the Markov chain $(R^*_i)_{i\in\N_0}$ defined in \eqref{Rdef},
$$
L_n^{\ssup 2}=\frac 1n\sum_{i=1}^n\delta_{(R^*_{i-1},R^*_i)}\in\Mcal_1(\Sigma^{\kappa-1}\times\Sigma^{\kappa-1}).
$$
In this expression, we assume periodic boundary conditions, i.e., $R^*_0=R^*_n$. Then $L_n^{\ssup 2}$ satisfies the marginal property: its two marginal measures are equal to each other. We denote  by $\Mcal_1^{\ssup {\rm s}}(\Sigma^{\kappa-1}\times\Sigma^{\kappa-1})$ the set of probability measures $\nu$ on $\Sigma^{\kappa-1}\times\Sigma^{\kappa-1}$ that satisfy this marginal property and write $\overline\nu$ for any of the two marginal measures of $\nu$. (The assumption $R^*_0=R^*_n$ is only of technical nature and can also be dropped without any problem, but we will not elaborate on that minor point.)

Since $(R^*_i)_{i\in\N_0}$ satisfies the condition (U) by Lemmas~\ref{lem-UCSMA} and \ref{lem-UALOHA}, respectively, \cite[Exercise 4.1.48]{deuschel} says that there is an invariant distribution $\nu_*$ of $(R^*_i)_{i\in\N_0}$. Then, by \cite[Lemma 4.1.45]{deuschel} the empirical pair measures $(L_n^{\ssup 2})_{n\in\mathbb{N}}$ converges almost surely towards $\bar{\nu}_*\otimes P_*$. Furthermore, we even get  a good control on the rate of this convergence: By \cite[Cor.~6.5.10 and Th.~6.5.12]{dembozeitouni} the empirical pair measures $(L_n^{\ssup 2})_{n\in\N}$ satisfies an LDP on $\Mcal_1(\Sigma^{\kappa-1}\times\Sigma^{\kappa-1})$ with rate function
\begin{equation}\label{entropy}
\nu\mapsto H(\nu\mid\overline\nu\otimes P_{*})= \int_{\Sigma\times\Sigma}\nu(\d R,\d R')\log\frac{\nu(\d R,\d R')}{\overline\nu(\d R)P_{*}(R,\d R')},
\end{equation}
if $\nu\in \Mcal_1^{\ssup {\rm s}}(\Sigma^{\kappa-1}\times\Sigma^{\kappa-1})$ is absolutely continuous with respect to $\overline \nu\otimes P_*$, and $\infty$ otherwise. The term in \eqref{entropy} is called the {\em relative entropy} of $\nu$ with respect to $\overline{\nu}\otimes P_*$. 

The empirical pair measures $L_n^{\ssup 2}$ stand in a simple one-to-one correspondence with the empirical $\kappa$-string measures that we are going to introduce now; we  would like to formulate  our LDP in terms of these measures instead. Consider the set $\Mcal_1^{\ssup{\rm s}}(\Sigma^\kappa)$ of probability measures $\mu$ on $\Sigma^\kappa$ whose first marginal measure $\mu^{\ssup{\kappa-1}}$ on $\Sigma^{\kappa-1}$ (i.e., when projected on the first $\kappa-1$ components) is equal to its second marginal measure (i.e., when projected on the last $\kappa-1$ components). Then the above LDP for $(L_n^{\ssup 2})_{n\in\N}$ is equivalent to saying that the empirical $\kappa$-string measure
$$
L_n^{\kappa}=\frac 1n\sum_{i=0}^{n-1}\delta_{((A_{i+1},\sigma_{i+1}),\dots,(A_{i+\kappa},\sigma_{i+\kappa}))}
$$
satisfies an LDP on $\Mcal_1(\Sigma^\kappa)$ with rate function
\begin{equation}\label{entropytwo}
\mu\mapsto H(\mu\mid \mu^{\ssup{\kappa-1}}\otimes W_{*})=\int_{\Sigma^{\kappa}}\d\mu \,\log\frac{\d\mu}{\d(\mu^{\ssup{\kappa-1}}\otimes W_{*})},
\end{equation}
if $\mu\in \Mcal_1^{\ssup{\rm s}}(\Sigma^\kappa)$, and $=\infty$ otherwise. 

Let us first analyse the rate function.  Recall the projections $\pi_1\colon \Sigma^\kappa \to \N$ and  $\pi_2\colon \Sigma^\kappa \to(0,\infty)$ defined by $\pi_1((a_1,r_1),\dots,(a_\kappa,r_\kappa))=a_\kappa$ and $\pi_2((a_1,r_1),\dots,(a_\kappa,r_\kappa))=r_\kappa$. As a prestep, we analyse a candidate for the rate function in an LDP for the pair $(\langle \pi_1,L_n^\kappa\rangle, \langle \pi_1,L_n^\kappa\rangle)$.

\begin{lemma}[An auxiliary rate function]\label{lemJ}
For $*\in \{{\rm CSMA}, {\rm ALOHA}\}$, introduce $J_*\colon [0,\infty)^2\to[0,\infty)$ as
\begin{equation}\label{J*def}
J_*(x,y)=\sup_{A\in\mathbb{R}, B\in(-\infty,\lambda)}\inf_{\mu\in\mathcal{M}_1^{\ssup{s}}(\Sigma^{\kappa})}\big[A(x-\left<\pi_1,\mu\right>)+B(y-\left<\pi_2,\mu\right>)+H(\mu\mid \mu^{\ssup{\kappa-1}}\otimes W_*)\big].
\end{equation}
Then $J_*$ is convex and hence continuous and possesses precisely one minimizer $(x_{\rm min},y_{\rm min})\in(0,\infty)^2$. 
\end{lemma}

\begin{proof}
By \cite[Theorems 4.1.43 and Lemma 4.1.45]{deuschel}, the map $\nu\mapsto H(\nu\mid\overline{\nu}\otimes P_*)$ is convex and possesses the unique minimizer $\overline{\nu}_*\otimes P_*$, where $\nu_*$ is the invariant distribution of $(R^*_i)_{i\in\N_0}$, whose existence is implied by Condition (U). Hence, the map $\mu\mapsto H(\mu\mid\mu^{\ssup{\kappa-1}}\otimes W_*)$ is also convex with the only minimizer $\widetilde{\mu}_*:=\nu^{\ssup{\kappa-1}}_*\otimes W_*\in \Mcal_1^{\ssup{\rm s}}(\Sigma^\kappa)$. Hence, $J_*$ has a unique minimizer, which is equal to $(x_{\rm min},y_{\rm min})=(\langle\pi_1, \widetilde{\mu}_*\rangle,\langle \pi_2,\widetilde{\mu}_*\rangle)$.  It is an easy exercise to prove the convexity of $J_*$, using \eqref{J*def} (observe that it is the Legendre transform of a function that is an supremum of linear functions). In particular, it is continuous in $(0,\infty)^2$, and all the right- and the left partial derivatives exist in $(0,\infty)^2$.
\end{proof}

Recall from Theorem~\ref{thm-LDP} the rate functions
\begin{equation*}
I_{\rm CSMA}(a,s)=\sup_{A\in\R,B\in(-\infty,\lambda)}\inf_{\mu\in\Mcal_1^{\ssup{\rm s}}(\Sigma^\kappa)}\big[A\left(a-s\left<\pi_1,\mu\right>\right)+B\left(1-s\left<\pi_2,\mu\right>\right)+H(\mu\mid \mu^{\ssup{\kappa-1}}\otimes W_{\rm CSMA})\big]
\end{equation*} 
for the CSMA protocol and
\begin{equation*}
\begin{aligned}
I_{\rm ALOHA}(a,s)=\sup_{A\in\R,B\in(-\infty,\lambda)}\inf_{\mu\in\Mcal_1^{\ssup{\rm s}}(\Sigma^\kappa)}\Big[A\big(a-\smfrac{a+s}{2}\left<\pi_1,\mu\right>\big)&+B\big(1-\smfrac{a+s}{2}\left<\pi_2,\mu\right>\big)\\
&+H(\mu\mid \mu^{\ssup{\kappa-1}}\otimes W_{\rm ALOHA})\Big]
\end{aligned}
\end{equation*} 
in the case of ALOHA protocol, where $a,s\in[0,\infty)$.

\begin{cor}[Properties of $I_{\rm CSMA}$ and $I_{\rm ALOHA}$]\label{lem-PropRFs}For both $*\in\{{\rm CSMA}, {\rm ALOHA}\}$,  $I_{*}$ is convex and hence continuous in $(0,\infty)^2$ and has precisely one minimizer $(a_{\rm min},s_{\rm min})\in(0,\infty)^2$ (given in \eqref{min_CSMA} and \eqref{min_ALOHA}, respectively). 
\end{cor}

\begin{proof} Observe that, for any $a,s\in[0,\infty)$,
$$
I_{\rm CSMA}(a,s)=sJ_{\rm CSMA}(\smfrac as,\smfrac 1s)\qquad \mbox{and}\qquad I_{\rm ALOHA}(a,s)=\frac {a+s}2 J_{\rm ALOHA}(\smfrac {2a}{a+s},\smfrac 2{a+s}).
$$
Using this, one easily sees that $I_*$ is uniquely minimized in the points 
\begin{align}\label{min_CSMA}
(a_{\rm min},s_{\rm min})=\left(\frac{x_{\rm min}}{y_{\rm min}},\frac{1}{y_{\rm min}}\right)=\left(\frac{\left<\pi_1,\tilde{\nu}\right>}{\left<\pi_2,\tilde{\nu}\right>},\frac{1}{\left<\pi_2,\tilde{\nu}\right>}\right)
\end{align}
in the CSMA case and 
\begin{align}\label{min_ALOHA}
(a_{\rm min},s_{\rm min})=\left(\frac{x_{\rm min}}{y_{\rm min}},\frac{2-x_{\rm min}}{y_{\rm min}}\right)=\left(\frac{\left<\pi_1,\tilde{\nu}\right>}{\left<\pi_2,\tilde{\nu}\right>},\frac{2-\left<\pi_1,\tilde{\nu}\right>}{\left<\pi_2,\tilde{\nu}\right>}\right)
\end{align}
in the ALOHA case.

For showing the convexity of $I_{\rm CSMA}$ we need to show that $(a,s)\mapsto s f(\frac as,\frac 1s)$ is convex if $f$ is convex. Fix $(a_1,s_1), (a_2,s_2)\in[0,\infty)^2$, then we see that
$$
f\big(\smfrac{a_1+a_2}{s_1+s_2},\smfrac 1{s_1+s_2}\big)
=f\Big(\smfrac {s_1}{s_1+s_2}\big(\smfrac {a_1}{s_1},\smfrac 1{s_1}\big)+\smfrac {s_2}{s_1+s_2}\big(\smfrac {a_2}{s_2},\smfrac 1{s_2}\big)\Big)
\leq \smfrac{s_1}{s_1+s_2} f\big(\smfrac {a_1}{s_1},\smfrac 1{s_1}\big)+ \smfrac{s_2}{s_1+s_2} f\big(\smfrac {a_2}{s_2},\smfrac 1{s_2}\big).
$$
Multiplying with $\frac 12(s_1+s_2)$ implies the convexity of $(a,s)\mapsto s f(\frac as,\frac 1s)$. A similar proof shows the convexity of $I_{\rm ALOHA}$. Indeed, again assume that $f$ is convex and pick $(a_1,s_1), (a_2,s_2)\in[0,\infty)^2$, then we see that
$$
\begin{aligned}
f\Big(\smfrac{a_1+a_2}{\frac 12(a_1+a_2+s_1+s_2)}, \smfrac 2{\frac 12(a_1+a_2+s_1+s_2)}\Big)
&\leq \frac{a_1+s_1}{a_1+a_2+s_1+s_2} f\Big(\smfrac{a_1}{\frac 12 (a_1+s_1)},\frac 2{\frac 12 (a_1+s_1)}\Big)\\
&\quad+ \frac{a_2+s_2}{a_1+a_2+s_1+s_2}f\Big(\smfrac{a_2}{\frac 12 (a_2+s_2)},\frac 2{\frac 12 (a_2+s_2)}\Big).
\end{aligned}
$$
Multyplying with $\frac 12 (a_1+a_2+s_1+s_2)$ implies the convexity of 
 $(a,s)\mapsto \frac {a+s}2 f(\frac {2a}{a+s},\frac 2{a+s})$ and hence the one of $I_{\rm ALOHA}$.
\end{proof}

Now we prove the large-deviation upper bound for $(\langle \pi_1,L_n^\kappa\rangle,\langle \pi_2,L_n^\kappa\rangle)$.

\begin{lemma}[LDP upper bound for $(\langle \pi_1,L_n^\kappa\rangle,\langle \pi_2,L_n^\kappa\rangle)$]\label{lem-LDPone}
For both $*\in\{{\rm CSMA}, {\rm ALOHA}\}$, for any closed set $F\subset (0,\infty)^2$,
\begin{align*}
\limsup_{n\to\infty}\frac{1}{n}\log\mathbb{P}\big((\langle \pi_1,L_n^\kappa\rangle,\langle \pi_2,L_n^\kappa\rangle)\in F\big)\leq -\inf_F J_*.
\end{align*}
\end{lemma}

\begin{proof} We are going to apply the G\"artner--Ellis theorem, which implies our assertion if the function $\Lambda$, defined by
\begin{equation}\label{Lambdadef}
\Lambda(A,B)=\lim_{n\to\infty}\frac{1}{n}\log\mathbb{E}\left[\e^{n(A\left<\pi_1,L_n^{\kappa}\right>+B\left<\pi_2,L_n^{\kappa}\right>)}\right],\qquad A\in\R,B\in(-\infty,\lambda),
\end{equation}
exists and is lower semi-continuous. In this case, the rate function (which is convex and lower semi-continuous) is given by the Legendre transform of $\Lambda$, which we will identify as the function $J_*$ defined in \eqref{J*def}. 

If $\Lambda$ would be differentiable, then the G\"artner--Ellis theorem would provide also the corresponding lower bound and hence a full LDP. However, in our case we do not know if this is true, due to the discontinuity of the mappings $\mu\mapsto \langle \pi_i,\mu\rangle$ for $i\in\{1,2\}$ in the weak topology of probability measures, since $\pi_i$ is not bounded.

So let us identify the limit in \eqref{Lambdadef}, which will be done with the help of the LDP for $(L_n^\kappa)_{n\in\N}$. Also here, we are facing the serious problem of missing unboundedness of $\pi_1$ and $\pi_2$; but we found a way around this. Indeed, we absorb the $\pi_2$-part in the transition kernel and employ a cutting argument for the $\pi_1$-part. We write now $\E_*=\mathbb{E}_{*}^{\ssup\lambda}$ for the expectation with respect to our Markov chain (stressing the arrival parameter $\lambda$ of the underlying PPP). The following trick absorbs the $\pi_2$-integral into the transition kernel. For this, we write $W_*^{\ssup\lambda}=W_*$ to stress the parameter $\lambda$ in the transition kernel of our Markov kernel, and we introduce the transformed kernels
\begin{equation}\label{W*AB}
W_{*}^{\ssup{A,B,\lambda}}((a,t),(k,\d s))=\e^{Ak+Bs}W_*((a,t),(k,\d s)),\qquad A\in\R,B\in(-\infty,\lambda), s\in(0,\infty),k\in\N.
\end{equation}
Then we observe that
\begin{equation}\label{W*ABscaling}
W_{*}^{\ssup{A,B,\lambda}}=W_{*}^{\ssup{D,0,\lambda-B}},\qquad \mbox{where }D:=A+\log\frac{\lambda}{\lambda-B}.
\end{equation}
As a consequence,
\begin{equation}\label{Hscaling}
H(\mu\mid \mu^{(\kappa-1)}\otimes W_*^{\ssup{\lambda-B}})
=\log\frac{\lambda}{\lambda-B}\left<\pi_1,\mu\right>-B\left<\pi_2,\mu\right>+H(\mu\mid \mu^{\ssup{\kappa-1}}\otimes W_*).
\end{equation}

We have from \eqref{W*ABscaling} that
$$
\mathbb{E}_*^{\lambda}\left[\e^{n(A\left<\pi_1,L_n^{\kappa}\right>+B\left<\pi_2,L_n^{\kappa}\right>)}\right]=\mathbb{E}_*^{\ssup{\lambda-B}}\left[\e^{n D\left<\pi_1,L_n^{\kappa}\right>} \right],\qquad n\in\N.
$$
For definiteness, assume that $D>0$; the opposite case is almost the same. Since we can lower bound $\pi\geq \pi\wedge m$ and since $\mu\mapsto \left<\pi_1\wedge m, \mu\right>$ is continuous and since $(L_n^{\kappa})_{n\in\mathbb{N}}$ satisfies an LDP with rate function $\mu\mapsto H(\mu\mid \mu^{\ssup{\kappa-1}}\otimes W_*)$, Varadhan's lemma tells us that
\begin{equation}\label{lowboundexpmom}
\liminf_{n\to\infty}\frac{1}{n}\log\mathbb{E}_*^{\ssup{\lambda-B}}\left[\e^{nD \left<\pi_1,L_n^{\kappa}\right>}\right]
\geq \lim_{m\to\infty}\sup_{\mu}\left[D \left<\pi_1\wedge m,\mu\right>-H(\mu\mid \mu^{\ssup{\kappa-1}}\otimes W_*^{\ssup{\lambda-B}})\right]=M(A,B),
\end{equation}
where
\begin{equation}\label{Mdef}
M(A,B)=\sup_{\mu}\left[A\left<\pi_1,\mu\right>+B\left<\pi_2,\mu\right>-H(\mu\mid \mu^{(\kappa-1)}\otimes W_*)\right],
\end{equation}
where we used \eqref{Hscaling}. We will show in the following that also the complementary inequality to \eqref{lowboundexpmom} holds, which shows that \eqref{Lambdadef} holds with $\Lambda=M$. This finishes the proof of the lemma, since it easily follows from \eqref{Mdef} that $J_*$ defined in \eqref{J*def} is the Legendre transform of $M=\Lambda$. 

For estimating in the opposite direction, we need to employ a cutting argument as follows. For any $m\in\N$ and $\delta>0$,
$$
\begin{aligned}
\mathbb{E}_*^{\ssup{\lambda-B}}\left[\e^{n D\left<\pi_1,L_n^{\kappa}\right>} \right]
&=\mathbb{E}_*^{\ssup{\lambda-B}}\left[\e^{n D\left<\pi_1,L_n^{\kappa}\right>} \1_{\left\{D|\left<\pi_1-\pi_1\wedge m,L_n^{\kappa}\right>| \leq\delta \right\}}+\e^{n D\left<\pi_1,L_n^{\kappa}\right>} \1_{\left\{D|\left<\pi_1-\pi_1\wedge m,L_n^{\kappa}\right>|>\delta \right\}}\right]
\\
&\leq \e^{\delta n}\mathbb{E}_*^{\ssup{\lambda-B}}\left[\e^{nD \left<\pi_1\wedge m,L_n^{\kappa}\right>}\right]
+\mathbb{E}_*^{\ssup{\lambda-B}}\left[\e^{n D\left<\pi_1,L_n^{\kappa}\right>} \1_{\left\{D\sum_{i=1}^n(A_i-m)_+>\delta n\right\}}\right].
\end{aligned}
$$
We need to show that the exponential large-$n$ rate of the last term is arbitrarily small if $m$ is picked large, then the complementary inequality to \eqref{lowboundexpmom} follows. From Lemmas~\ref{lem-MarkovCSMA} and \ref{lem-MarkovALOHA}, respectively, we know that, given $(\sigma_i)_{i\in\N}$, under $\E^{\ssup{\lambda-B}}$, the random variables $A_1,\dots,A_n$ are independent and have the distribution of $1+$ a $\Poi_{\lambda_i}$-distributed random variable, where $\lambda_i=(\lambda-B) \gamma_i$ in the CSMA-case and $\lambda_i=(\lambda-B)( \gamma_i+B(s)/\kappa)$ in the ALOHA-case, where $\gamma_i=(1-\sum_{j=1}^{\kappa-1}\sigma_{i-j})_+$. In any case, we have $\lambda_i\leq 2\lambda$ for any $i$ and can therefore estimate
$$
\P_*^{\ssup{\lambda-B}}\big(A_i=k\,\big|\, (\sigma_j)_{j\in\N}\big)\leq \e^{2\lambda}\Poi_{2\lambda}(k-1),\qquad k\in\N.
$$
Hence, we can estimate, writing $\E$ for expectation with respect to $\Poi_{2\lambda}$, using the exponential Chebyshev inequality with some $K>0$,
\begin{equation}\label{LDPproofCheby}
\begin{aligned}
\mathbb{E}_*^{\ssup{\lambda-B}}\left[\e^{n D\left<\pi_1,L_n^{\kappa}\right>} \1_{\left\{D\sum_{i=1}^n(A_i-m)_+>\delta n\right\}}\right]
&\leq \e^{2\lambda n} \E\left[\e^{n D\left<\pi_1,L_n^{\kappa}\right>} \1_{\left\{D\sum_{i=1}^n(A_i-m)_+>\delta n\right\}}\right]\\
&\leq \e^{2\lambda n} \e^{-\frac{K\delta n}{D}}\E\left[\e^{D\sum_{i=1}^n A_i}\,\e^{K\sum_{i=1}^n (A_i-m)_+}\right]\\
&=\exp\Big\{-n\Big[-\log (2\lambda)+K\frac \delta D-\log \E\big[\e^{D A_1+K(A_1-m)_+}\big]\Big]\Big\}.
\end{aligned}
\end{equation}
We need to show that the last term in the brackets can be made arbitrarily large as $m\to\infty$ with an appropriate choice of $K=K(m)$. We estimate the last term as follows:
$$
\begin{aligned}
\E\big[\e^{D A_1+K(A_1-m)_+}\big]&=\sum_{k=0}^m\e^{-2\lambda}\frac {(2 \lambda)^k}{k!} \e^{D k}+\sum_{k>m}\e^{-2\lambda}\frac {(2\lambda)^k}{k!} \e^{D k}\e^{K(k-m)}\\
&\leq \e^{2\lambda(\e^{D}-1)} +\frac{(2\lambda \e^D)^m}{m!}\e^{2\lambda(\e^{D+K}-1)},
\end{aligned}
$$
where we used an index shift and estimated $\frac 1{(k+m)!}\leq \frac1{k!}\,\frac 1{m!}$. Now it is easy to see that one can pick $K=K(m)\to\infty$ in such a way that the term in the brackets on the right of \eqref{LDPproofCheby} diverges to $\infty$ (take $K$ of order $\log m$).

This finishes the proof of the complementary inequality to \eqref{lowboundexpmom} and therefore finishes the proof of the lemma.

\end{proof}

Now we can prove the main result of this section, the LDP upper bounds for $\frac 1t(A(t),S(t))$. This finishes the proof of Theorem~\ref{thm-LDP}.  

\begin{lemma}\label{LDPCSMA} For $*\in\{{\rm ALOHA, CSMA}\}$, as $t\to\infty$, the pair $\frac 1t (A(t),S(t))$ satisfies the LDP upper bound on $(0,\infty)^2$ with rate function $I_*$, i.e., for any closed set $F\subset (0,\infty)^2$ we have
\begin{align*}
\limsup_{t\to\infty}\frac{1}{t}\log\mathbb{P}\Big(\frac{1}{t}\big(S(t),A(t)\big)\in F\Big)\leq -\inf_F I_*
\end{align*}
where $I_*$ is given before Lemma~\ref{lem-PropRFs}. Furthermore, in both cases $(\frac 1t (A(t),S(t)))_{t>0}$ is exponentially tight, i.e., for any $M>0$ there is an $K>0$ such that $\P( \frac 1t (A(t),S(t))\in ([0,K]^2)^{\rm c})\leq \e^{-Mt}$ for any large $t$.
\end{lemma}

\begin{proof} Let us explain our proof strategy. As we already mentioned above, the sequence $(L_n^{\ssup2})_{n\in\N}$ of empirical pair measures of the Markov chain $(R_i^*)_{i\in\N_0}$ satisfies an LDP with rate function given in \eqref{entropy}. Then it is clear that the sequence $(L_n^\kappa)_{n\in\N}$ of empirical $\kappa$-string measures satisfies an LDP with rate function given in \eqref{entropytwo}. It will turn out that $(A(t),S(t))$ can be expressed in terms of the partial sums 
\begin{equation}\label{timeinverses}
\sum_{i=1}^n A_i=n\langle \pi_1,L_n^\kappa\rangle\qquad\mbox{and}\qquad \sum_{i=1}^n \sigma_i=n\langle \pi_2,L_n^\kappa\rangle.
\end{equation}
We will derive the upper bound of the LDP for $\frac 1t (A(t),S(t))$ from this, in combination with the upper bound of Lemma~\ref{lem-LDPone} for the pair $(\langle \pi_1,L_n^\kappa\rangle,\langle \pi_2,L_n^\kappa\rangle)$. This derivation will be heuristically done in Step 1 of the proof. Since $(A(t),S(t))$ is basically a pair of time-inverses, probabilities of events of one-sided inequalities like $\{A(t)<a,S(t)>s\}$ for $a,s\in(0,\infty)$ are relatively easy to handle, and this we will do in Step 2. The LDP upper bound for $\frac 1t (A(t),S(t))$ will be proved in Step 3, while the proof of  the exponential tightness is contained in Step 2.

\medskip

{\bf Step 1: Heuristics.} We assume that $(\langle \pi_1,L_n^\kappa\rangle,\langle \pi_2,L_n^\kappa\rangle)$ satisfies the LDP with rate function $J_{*}$ given in \eqref{J*def} and derive heuristically the LDP for $\frac 1t(A(t),S(t))$ from that. Let us first treat the CSMA protocol. Fix $a,s>0$.  Using \eqref{timeinverses}, we see (ignoring that $at$ and $st$ may be not integers) that, as $t\to\infty$,
\begin{equation}\label{S(t)CSMA}
\begin{aligned}
\mathbb{P}\left(S(t)\approx st, A(t)\approx at\right)&=\mathbb{P}\Big(\sum_{i=1}^{st}\sigma_i\approx t, \sum_{i=1}^{st}A_i\approx at\Big) = \mathbb{P}\Big(\langle\pi_2, L_{st}^{\kappa}\rangle\approx\frac{1}{s}, \langle\pi_1, L_{st}^{\kappa}\rangle\approx\frac{a}{s}\Big)
\\
&\approx\exp\big(-st J_{\rm CSMA}(a/s, 1/s)\big)=\exp\big(-tI_{\rm CSMA}(a,s)\big).
\end{aligned}
\end{equation}
This finishes the heuristics for the CSMA case.

In the ALOHA case, let $\widetilde{S}(t)$ be number of potentially successful messages that arrive by time $t$, i.e., those that pick a free channel at arrival. Since every time that a new arriving message picks a busy channel, the old one that is already in this channel also gets lost, the number of successfully delivered messages by time $t$ is obtained by subtracting the number of new arriving messages taking a busy channel, from the number of potentially successful messages. The former is equal to $A_i-1$ in each interval $(\widetilde{T}_i, \widetilde{T}_{i+1}]$. Considering $\widetilde{S}(t)$ potentially successful messages and therefore $\widetilde{S}(t)$ intervals, it means that we have
\begin{equation}\label{S(t)Aloha}
S(t)=\widetilde{S}(t)-\sum_{i=1}^{\widetilde{S}(t)}(A_i-1)=\widetilde{S}(t)-\Big(-\widetilde{S}(t)+\sum_{i=1}^{\widetilde{S}(t)}A_i\Big)=2\widetilde{S}(t)-\sum_{i=1}^{\widetilde{S}(t)}A_i\approx 2\widetilde{S}(t)-A(t).
\end{equation}
Furthermore, $\sum_{i=1}^{\widetilde{S}(t)}\sigma_i\approx t$, as we have already seen in the case of CSMA. Now, let $s,a>0$, the we have, as $t\to\infty$,
$$
\begin{aligned}
\mathbb{P}(S(t)=st, A(t)=at)&\approx\mathbb{P}\big(\widetilde{S}(t)\approx\smfrac{1}{2}(a+s)t, A(t)=at\big)
\\
&=\mathbb{P}\Big(\sum_{i=1}^{\frac{1}{2}(s+a)t}\sigma_i\approx t, \sum_{i=1}^{\frac{1}{2}(s+a)t}A_i\approx at\Big)
%\\=&\mathbb{P}\left(\frac{2}{(s+a)t}\sum_{i=1}^{\frac{1}{2}(s+a)t}\sigma_i\approx \frac{2}{s+a}, \frac{2}{(s+a)t}\sum_{i=1}^{\frac{1}{2}(s+a)t}A_i\approx\frac{2a}{s+a}\right)
\\
&\approx \mathbb{P}\Big(\big\langle\pi_2, L^{\kappa}_{\frac{1}{2}(s+a)t}\big\rangle=\frac{2}{s+a}, \big\langle\pi_1, L^{\kappa}_{\frac{1}{2}(s+a)t}\big\rangle=\frac{2a}{s+a}\Big)
\\
&\approx\exp\Big(-t\frac {a+s}2 J_{\rm ALOHA}(\smfrac{2a}{a+s}, \smfrac 2{a+s})\Big)=\exp\big(-tI_{\rm ALOHA}(a,s)\big), 
\end{aligned}
$$
which finishes the heuristics.

\medskip

{\bf Step 2: Exponential rates for quadrants.}
As we mentioned, one-sided inequalities for $S(t)$ and $A(t)$ are relatively easily to handle. We demonstrate this by showing, as a first step, the exponential tightness of $(\frac 1t(A(t),S(t)))_{t>0}$. Indeed, for any $K>0$, 
$$
\P\big({\smfrac 1t}(A(t),S(t))\in \big([0,K]^2\big)^{\rm c}\big)
\leq \P(A(t)> t K)+\P(S(t)> t K).
$$
It is easy to see that $\lim_{K\to\infty}\limsup_{t\to\infty}\frac 1t \log \P(A(t)> t K)=-\infty$, observing that $A(t)$ is $\Poi_{\lambda t}$-distributed. Now, using the LDP for $(L_n^{\kappa})_{n\in\mathbb{N}}$ and again using the G\"artner--Ellis theorem we can derive an LDP upper bound for $(\left<\pi_2,L_n^{\kappa}\right>)$ as in Lemma \ref{lem-LDPone} with rate function
\begin{align*}
\tilde{J}(y)=\sup_{B\in(-\infty, \lambda)}\inf_{\mu\in\mathcal{M}_1^{\ssup s}(\Sigma^{\kappa-1})}\left[B(y-\left<\pi_2,\mu\right>)+H(\mu\mid \mu^{\ssup \kappa-1}\otimes W_*)\right].
\end{align*}
Then, using also \eqref{timeinverses}, we estimate
$$
\P(S(t)> t K)=\P\Big( \sum_{i=1}^{tK}\sigma_i<t\Big)
\leq \P\big(\langle \pi_2,L_{tK}^\kappa\rangle\leq \smfrac 1K\big)
\leq \e^{-tK \inf_{(0,1/K]}\tilde{J}}\e^{o(t)},
$$. 
Then, it is easy to see that $K \inf_{(0,1/K]}\tilde{J}\to\infty $ as $K\to\infty$, hence exponential tightness follows.

We further use the simple relation between the partial sum of the $\sigma_i$'s and $(A(t),S(t))$ for proving that, for any $(a,s)\in(0,\infty)^2$,
\begin{equation}\label{LDPquadrant}
\partial_aI(a,s)>0,\partial_sI(a,s)>0\qquad\Longrightarrow\qquad \limsup_{t\to\infty}\frac 1t\log\P\big(\smfrac 1t(A(t),S(t)\big)\in [a,\infty)\times [s,\infty)\big)\leq -I_*(a,s).
\end{equation}
Analogous statements for all the other sign combinations of the partial derivatives with the respective quadrants are also true and are proved in the same manner; we omit these proofs. Since $I_*$ is continuous, we can freely replace the closed set $[a,\infty)\times [s,\infty)$ by $(a,\infty)\times (s,\infty)$.

We prove now \eqref{LDPquadrant} for the CSMA case; the other one is similar and will be omitted.
Fix $(a,s)$ such that $\partial_aI(a,s)$ and $\partial_sI(a,s)$ are both positive. For showing \eqref{LDPquadrant}, we see that (again using \eqref{timeinverses})
\begin{align*}
\P\big(\smfrac 1t(A(t),S(t))\in [a,\infty)\times [s,\infty)\big) &=\P\big(S(t)\geq st,A(t)\geq at\big)\\
&=\P\Big(\langle \pi_2,L_{ts}^\kappa\rangle \leq \frac 1{s}, 
\langle \pi_1,L_{ts}^\kappa\rangle \geq\frac{a}{s}\Big)\\
&\leq \exp\Big(-ts\inf_{[\smfrac {a}{s},\infty)\times (-\infty,\frac 1 {s}]}J_{{\rm CSMA}}\Big)\e^{o(t)}
\\&=\exp\Big(-t \inf_{[a,\infty)\times [s,\infty)}I_{\rm CSMA}\Big)\,\e^{o(t)}= \exp\big(-t I_{\rm CSMA}(a,s)\big)\,\e^{o(t)},
\end{align*}
where in the estimate we used the upper bound in the LDP for $L_{ts}^\kappa$ and afterwards that  $I_{\rm CSMA}(a,s)=s J_{\rm CSMA}(\frac as,\frac 1s)$ and then that it is continuous and assumes its infimum over $[a,\infty)\times [s,\infty)$ in the corner of this quadrant. The latter comes from the convexity of $I_{\rm CSMA}$ and the positivity of the two partial derivatives. 
 
For handling the case that one of the two partial derivatives is zero, we claim that
\begin{equation}\label{LDPhalfplane}
\partial_aI(a,s)=0,\partial_sI(a,s)>0\qquad\Longrightarrow\qquad \limsup_{t\to\infty}\frac 1t\log\P\big(\smfrac 1t(A(t),S(t)\big)\in [a,\infty)\times [s,\infty)\big)\leq -I_*(a,s).
\end{equation} 
The proof of this is  similar to the proof of \eqref{LDPquadrant}, but estimates against the half plane $[0,\infty)\times [s,\infty)$ and uses that the infimum of $I_*$ over $[0,\infty)\times [s,\infty)$ is attained at $(a,s)$ by convexity of $\widetilde a\mapsto I_*(\widetilde a,s)$ and because of $\partial_aI(a,s)=0$. We omit the details.

\medskip

{\bf  Step 3: Proof of the upper bound.} We use the fact that an  exponentially tight sequence $(X_t)_{t>0}$ of random variables satisfies the LDP upper bound with rate function $I$ on a Polish space $\mathcal X$  if 
\begin{equation}\label{LDPcriterion}
\lim_{\eps\downarrow0}\lim_{t\to\infty}\frac 1t \log \P(X_t\in B_\eps(x))\leq -I(x),\qquad x\in\mathcal X.
\end{equation}
The proof of this fact is an elementary exercise using standard compactness arguments; we omit the proof.
 
We now check \eqref{LDPcriterion}. Fix $x=(a,s)\in (0,\infty)^2$. Let us first consider the case that $\partial_a I(a,s)$ and $\partial_s I(a,s)$ are both not equal to zero. Let us assume, for definiteness, that $\partial_aI(a,s)>0$ and $\partial_s I(a,s)>0$. Pick an $\eps>0$ such that the two partial derivatives are positive inside $B_\eps(a)\times B_\eps(s)$. Then we can estimate from above, as $t\to\infty$, according to \eqref{LDPquadrant},
$$
\P\big(\smfrac 1t(A(t),S(t))\in B_\eps(a)\times B_\eps(s)\big)
\leq \P\big(\smfrac 1t(A(t),S(t))\in [a-\eps,\infty)\times [s-\eps,\infty)\big)
\leq \e^{-t I_*(a-\eps,s-\eps)}\e^{o(t)}.
$$
Since $I_*$ is continuous, we see that  \eqref{LDPcriterion} is satisfied.

It remains to handle the case where one of the two partial derivatives is equal to zero. If both are, then $(a,s)$ is the unique minimal point $(a_{\rm min},s_{\rm min})$ of $I_*$, and the exponential rate is equal to zero, which is equal to $I_*(a_{\rm min},s_{\rm min})$. The remaining case that precisely one of the two partial derivatives vanishes, can be handled either by an approximation argument (using the continuity of $I_*$) or by appealing to \eqref{LDPhalfplane}; we omit the details.
\end{proof}

\end{document}